\documentclass[11pt]{amsart}

\usepackage{preamble}

\let\amp=&
\usepackage{makecell}

\begin{document}

\title{Automorphisms of locally string algebras}

\address{Department of Mathematics, Box 354350, University of Washington, Seattle, Washington 98195, USA}

\author{S. Ford}
\email{siford@uw.edu}

\begin{abstract}
It is known that automorphisms of finite-dimensional bound quiver algebras decompose into inner automorphisms and automorphisms which permute the vertices. In this paper, we show that for string algebras, automorphisms permuting  vertices further decompose into a graded automorphism and a composition of certain types of exponential automorphisms. Moreover, the same decomposition applies to automorphisms of locally string algebras, which are an analogue of string algebras wherein the finite-dimensional condition is omitted.
\end{abstract}



\maketitle
\section{Introduction}

The automorphism group of an algebra is hard to compute in general. The Launois-Lenagan and Andruskiewitsch-Dumas conjectures describing the automorphism groups of certain classes of algebra were proven by Yakimov in  \cite{Ya13,Ya14}. In this paper, we will describe the automorphisms of another class of algebra, the (locally) string algebras.

Monomial special biserial algebras, better known as string algebras, are a   well-studied class of finite-dimensional bound quiver algebra \cite{SW83, WW85, BR87, R94, Bl98, R01, HZS05, Crawley-Boevey18, LPP18}.
For a finite-dimensional bound quiver algebra $A=\kk\cQ/\cI$, it is known that $\Aut(A)=\Hhat{A}\cdot \Inn^*(A)$, where $\Hhat{A}=\setst{\phi\in\Aut(A)}{\phi\vert_{\kk\cQ_0}\subseteq {\kk\cQ_0}}$ and $\Inn^*(A)$ denotes the inner automorphisms of $A$ which correspond to elements of $\one+J(A)=\one+\langle \alpha\rangle_{\alpha\in\cQ_1}$ \cite{GS99, Po89}. It was further shown in \cite{GS99} that $H_A/H_A\cap \Inn^*(A)$ is a Morita invariant of $A$ and that the Picard group of $A$ may be computed from $H_A/H_A\cap \Inn^*(A)$, where $H_A=\setst{\phi\in\Aut(A)}{\phi\vert_{\kk\cQ_0}=\id_{\kk\cQ_0}}$. In this paper, we show that for a string algebra $A$, an element of $H_A$ may be decomposed into a graded automorphism, a composition of exponential automorphisms, and elements of $\Inn^*(A)$ (see Theorem \ref{thm:strA0}). In Theorem \ref{thm:MIgentle}, we use this to compute $H_A/H_A\cap \Inn^*(A)$.

The automorphism groups of infinite-dimensional algebras are particularly difficult to compute. We consider the special case of locally gentle algebras, which are infinite-dimensional analogues of string algebras. The same decomposition of automorphisms into exponential, graded and inner automorphisms holds for this case (see Theorem \ref{thm:decomp}). Consequently, if we redefine 
$\Inn^*(A)$ in terms of the graded Jacobson radical 
(noting that the definitions of $\Inn^*(A)$ coincide if $A$ is finite-dimensional), the property $\Aut(A)=\Hhat{A}\cdot \Inn^*(A)$ holds in this infinite-dimensional case as well.

Section \ref{sec:prelim} lays out the definition of a (locally) string algebra and recalls some basic results and notation from \cite{FOZ}, with some minor additions relevant to (locally) string algebras.

In section \ref{sec:ppoaut}, we introduce some preliminary properties of automorphisms of (locally) string algebras and introduces notation for key subgroups of the automorphism group. In particular, we focus on the subgroup $\Aut_0(A)$, which is subgroup of $f\in\Aut(A)$ for which the associated graded map $\gr f$ (where $\gr f(x)$ is the degree $\deg(x)$ component of $f(x)$ for homogeneous $x$) is the identity.
\begin{lemma}[Lemma \ref{lem:aut0}] Suppose $A=\kk\cQ/\cI$ is (locally) string. For each automorphism $f\in \Aut(A)$, the associated graded map $\gr f$ is an automorphism and
\[1\rightarrow \Aut_0(A)\hookrightarrow \Aut(A)\xrightarrow[f\mapsto \gr f]{} \grAut(A)\rightarrow 1\]
is a short exact sequence. 
Moreover, if $A\not\cong \kk[x]$, then $f(w)\in \gr f(w)+A_{>\ell(w)}$ for all $w\in\cB$.
	\end{lemma}

In section \ref{sec:exp}, we consider two forms of nilpotent derivations, which correspond to elements of $H_A\cap \Inn^*(A)$. Derivations of {type I} correspond to elements of $\alpha A\alpha$, where $\alpha$ is an arrow contained in a finite maximal path, and derivations of {type II} correspond to linear  combinations of left maximal paths (with some additional conditions).
Up to an exponential automorphism corresponding to a derivation of type II, any automorphism in $\Aut_0(A)$ takes $w$ to $\langle w\rangle$ (Lemma \ref{cor:fixAbar}). This allows us to focus on two key cases: string algebras and locally gentle algebras with a unique, infinite maximal paths (subalgebras of $\kk\widetilde{A}_n$). In the string case, we have the following result.
\begin{proposition}[Proposition \ref{prop:strfixv}]
    Suppose that $A$ is a string algebra and $\phi\in \Aut_0(A)\cap H_A$. Then, there exist $y\in A_+$, $\rho\in D(A)$, and a  derivation $\delta$ of type II such that  $\phi =  \exp(\delta)\circ \rho\circ \Delta_{\one+y}.$
\end{proposition}
Here, $D(A)$ is the subgroup of $\Aut_0(A)\cap H_A$ generated by exponential automorphisms corresponding to derivations of type I. These are the automorphisms which fix vertices and take each arrow $\alpha$ to $\alpha+\alpha A\alpha$ (see Lemma \ref{lem:p(z)type}). While $D(A)$ is not a subset of $\Inn^*(A)$, there may be nontrivial intersection, as seen in the following example.
\begin{example} [Example \ref{ex:innerD(a)}]
    Consider $A=\kk\cQ/\langle (ab)^2,(ba)^2\rangle$ where $\cQ$ is the quiver\; \qquad\quad .
    
\vspace{-.9cm}    \begin{flushright}
        \begin{tikzpicture}
            \unvtx{1}{0,0}
            \unvtx{2}{1,0}
            \draw[->](1) to[bend left] node[midway, above]{$a$} (2);
            \draw[->](2) to[bend left] node[midway, below]{$b$} (1);
        \end{tikzpicture}\quad\;\;\;\,
  \end{flushright}\vspace{-.3cm}
  
\noindent The automorphism $\phi\in D(A)$ given by $\phi(a)=a+aba$ and $\phi(w)=w$ for $w\in\cB\setminus \{a\}$ is not inner. The intersection $D(A)\cap \Inn^*(A)$ contains $\Delta_{\one-ab+ba}$ and thus is nontrivial.
\end{example}

The case in which $A$ has a unique, infinite maximal path is considered in section \ref{sec:ump}. We embed $A$ into the ring of matrix polynomial to show the following result.
\begin{proposition}[Proposition \ref{prop:1mcase}]
    Suppose that $A=\kk\cQ/\cI$ is a locally gentle algebra with a unique, infinite maximal path. Then either $A\cong \kk[x]$ or $\Aut_0(A)\subseteq \Inn^*(A)$. Consequently, if $A\not\cong \kk[x]$, then any automorphism of $A$ decomposes into an inner automorphism and a graded automorphism.
\end{proposition}

Our main result, and the computation of $H_A/H_A\cap \Inn^*(A)$ for $A$ gentle, is in section \ref{sec:main}.
\begin{theorem}[Theorem \ref{thm:decomp}] Suppose that $A\not\cong \kk[x]$ is a (locally) string algebra. Then, $\phi\in \Aut_0(A)$ may be decomposed into a composition of exponential automorphisms (corresponding to derivations of type I and II) and an inner automorphism in $\Inn^*(A)$. 

If $A$ is (locally) gentle, then the composition of exponential automorphisms is again exponential.
\end{theorem}
When $A$ is (locally) gentle, $H_A$ is generated by inner automorphisms, graded automorphisms which fix the vertices, and exponential automorphisms corresponding to finite maximal paths.

\section{Definitions and Preliminaries}\label{sec:prelim}

Throughout this paper, we work over the base field $\kk$. 
In this section, we give the definition of a (locally) string algebra and recall some basic results and notation from \cite{FOZ}, with some additions relevant for (locally) string algebras. 

A \emph{finite quiver}, $\cQ=(\cQ_0,\cQ_1,s,t)$, is a directed graph with a finite vertex set $\cQ_0$, a finite set of edges (called arrows) $\cQ_1$, and functions $s,t:\cQ_1\to\cQ_0$ which indicate the \emph{source} and \emph{target} of each arrow, respectively. Throughout, we assume that $\cQ$ has at least one arrow.

A \emph{path}, $p$, in a quiver is concatenation of arrows $p=\alpha_1\alpha_2\dotsc\alpha_n$ where $t(\alpha_i)=s(\alpha_{i+1})$ for every $1\leq i <n$. We say the path $p$ has \emph{length} $n$, and we denote this by $\ell(p)=n$. We denote the first arrow of $p$ by $F(p)$ and the last arrow of $p$ by $L(p)$, , i.e. $F(p)=\alpha_1$ and $L(p)=\alpha_n$. Further, we let $s(p)=s(F(p))$ and $t(p)=t(L(p))$. 

For every $v\in\cQ_0$, we let $\epsilon_v$ denote a stationary path at the vertex $v$, which has length 0 and $s(\epsilon_v)=t(\epsilon_v)\defeq v$. We let $F(\epsilon_v),L(\epsilon_v)\defeq \epsilon_v$.
If $p$ and $q$ are two paths so that $t(p)=s(q)$, then $pq$ is the path formed by concatenating these two paths. If we have paths $r,q',q''$ so that $p=q'rq''$, then we say that $r$ is a \emph{subpath} of $p$ and denote this $r\leq p$. We say that $r$ is an \emph{initial} (respectively, \emph{terminal}) subpath of $p$ if $q'$ ($q''$) is stationary.

Given a quiver $\cQ$, we can form its \emph{path algebra} $\kk\cQ$, whose basis consists of all paths in $\cQ$ and whose multiplication is given by concatenation of paths.
This path algebra is graded by path length. An ideal $\cI$ of $\kk\cQ$ is \textit{homogeneous} if it is generated by homogeneous elements, in which case the quotient $\kk\cQ/\cI$ is also graded by path length.
If the ideal $\cI$ is additionally generated by paths, then the quotient $\kk\cQ/\cI$ is a monomial algebra with a basis consisting of paths. 

\begin{defn}\label{def:string} An algebra $A=\kk\cQ/\cI$ is called a \emph{string algebra} if
\begin{enumerate}[(i)]
    \item the quiver $\cQ$ is finite, the underlying graph is connected, and $\indeg(v), \outdeg(v)\leq 2$ for each vertex $v\in\cQ_0$.
    \item whenever $s(\alpha)=t(\beta)=t(\beta')$ with $\beta\neq\beta'$, at least one of $\beta\alpha$ and $\beta'\alpha$ is in $\cI$.
    \item whenever $t(\alpha)=s(\beta)=s(\beta')$ with $\beta\neq \beta'$, at least one of  $\alpha\beta$ and $\alpha\beta'$ is in $\cI$.
    \item the ideal $\cI$ is generated by paths.
    \item the algebra $A$ is finite-dimensional (i.e. $\cI$ is admissible).
\end{enumerate}
We say the algebra $A$ is \emph{locally string} if it is infinite-dimensional and satisfies conditions (i)-(iv), and we write (locally) string to mean string or locally string. 

The algebra  $A$ is \emph{(locally) gentle algebra} if it satisfies the modified conditions
\begin{enumerate}[(i)]
     \item[(ii')] whenever $s(\alpha)=t(\beta)=t(\beta')$ with $\beta\neq\beta'$, \textit{exactly} one of $\beta\alpha$ and $\beta'\alpha$ is in $\cI$.
    \item[(iii')] whenever $t(\alpha)=s(\beta)=s(\beta')$ with $\beta\neq \beta'$, \textit{exactly} one of  $\alpha\beta$ and $\alpha\beta'$ is in $\cI$.
    \item[(vi')] the ideal $\cI$ is generated by paths \textit{of length 2}.
\end{enumerate}
\end{defn}

Any (locally) string algebra is a quotient of a (locally) gentle algebra by a monomial ideal.

\begin{notation}\label{not:algebra}
For the remainder of this paper, $\cQ$ denotes a finite, connected quiver, $\cI$ denotes an ideal in $\kk\cQ$ so that $\kk\cQ/\cI$ is (locally) string, $A$ denotes the (locally) string algebra $\kk\cQ/\cI$, and $\cB$ denotes the basis of $A$ consisting of paths.

We let $\cB_n$ denote the paths in $\cB$ of length $n$ and $A_n$ denote the $\kk$-span of $\cB_n$. Since $A$ is graded by path length, $A_n$ is the degree $n$ component of $A$ and $A=\bigoplus_{n\geq 0} A_n$. Similarly, $\cB_+$ denotes the nonstationary paths in $\cB$ and $A_+\defeq \bigoplus_{n>0} A_n=\spn_{\kk}\cB_+.$
\end{notation}

An immediate consequence of (ii) in  Definition \ref{def:string} is that for any path $p \in \cQ$, there is at most one arrow $\alpha$ so that $\alpha p\not\in\cI$. 
Similarly, there is at most one arrow $\beta\in\cQ_1$ so that $p\beta\not\in\cI$.

\begin{defn}\label{def:maxpaths} 
Let $p$ be a path in $\cB$. We say that
\begin{itemize}
\item $p$ is a \emph{right maximal} if for every $\alpha\in\cQ_1$, $p\alpha\in\cI$.
\item $p$ is \emph{left maximal} if for every $\alpha\in\cQ_1$, $\alpha p\in\cI$.
\item $p$ is a \emph{finite maximal path} if $p$ is both left and right maximal. 
\end{itemize}
\end{defn}

\begin{example}
    The path algebra $\kk K$ of the Kronecker quiver, $K=$ \begin{tikzpicture}
        \unvtx{1}{0,0}
        \unvtx{2}{1,0}
        \undbl{1}{2}
    \end{tikzpicture}\,, is the only (locally) gentle algebra in which two parallel arrows are both finite maximal paths.
\end{example}

\begin{defn}
   We say that a cycle $c\in \cB$ generates an \emph{infinite maximal path} if $c$ is a nonstationary, primitive cycle in $\cB$ such that $c^n\notin\cI$ for any $n\in\NN$. We denote the associated infinite maximal path by $c^\infty$ and, for $w\in\cB$, we write $w\leq c^\infty$ if there exists $n\geq 0$ such that $w\leq c^n$. 
\end{defn}

If $A$ is (locally) gentle, $\cI$ is generated by paths of length two and so it suffices that $c^2\notin\cI$.

Since $A$ is (locally) string and $c$ is primitive, if $c$ generates an infinite maximal path and there exists $c'\leq c^\infty$ such that $(c')^2\notin\cI$, then we must have $c\leq (c')^n$ for some $n\geq 0$. Consequently, if $c'$ also generates an infinite maximal path, then $c\leq (c')^\infty$ if and only if $c'\leq c^\infty$. 
\begin{defn}
    Suppose that $c$ and $c'$ generate infinite maximal paths. If $c\leq (c')^\infty$ (or equivalently, $c'\leq c^\infty$), we say that 
    $c^\infty$ and $(c')^\infty$ are the same infinite maximal path.
\end{defn}

The following result is stated for locally gentle algebras in  \cite{FOZ}, but the argument extends to a (locally) string algebra.
\begin{lemma}
    Let $A$ be (locally) string. For any $\alpha\in \cQ_1$, there is a path $\gamma_r(\alpha)$ which is  maximal among the paths in $\cB$ which (i) start with $\alpha$ and (ii) have no repeated arrows. 

If $\alpha A\alpha\neq 0$, then $\gamma_r(\alpha)$ is a cycle and $\gamma_r(\alpha)\alpha\notin\cI.$
\end{lemma}

When $A$ is locally gentle,  $\alpha A\alpha\neq 0$ implies that  $\gamma_r(\alpha)$ generates an infinite maximal path.

\begin{lemma}\cite{FOZ}\label{lem:ump}
Suppose that $A$ is (locally) gentle. Every arrow $\alpha\in\cQ_1$ is contained in a unique maximal path, denoted $\gamma_\alpha$. Moreover, if $\gamma_\alpha$ is infinite, then $\gamma_\alpha=(\gamma_r(\alpha))^\infty.$
\end{lemma}

If  there is a path $p\in\cB_+$ containing arrows $\alpha$ and $\beta$, then $\gamma_\alpha=\gamma_\beta$ and $p\leq \gamma_\alpha$. Thus, if $A$ is (locally) gentle, $\cB_+$ is partitioned into subpaths of maximal paths.

As any (locally) string algebra $A=\kk\cQ/\cI$ is a quotient of some (locally) gentle algebra $A'=\kk\cQ/\cJ$ (where $\cJ\subset \cI$ is generated by a subset of the length two paths in $\cI$),  
any arrow $\alpha$ in $\cQ_1$ is contained in a unique maximal path $\gamma_\alpha$ of $A'$.
In particular, if $\gamma_\alpha$ is infinite in $A'$, then either $\gamma_\alpha$ is an infinite maximal path in $A$ as well or some (finite) subpath $w\leq \gamma_\alpha$ is in $\cI$. In the latter case, there must exist $n\in\NN$ such that every length $n+1$ subpath of $\gamma_\alpha$ contains $w$ and so a path in $\cB$ containing $\alpha$ has length at most $n$. Thus we have the following corollary.

\begin{corollary}\label{cor:mpstr}
    If $A$ is (locally) string, then every arrow is contained in a maximal path, not necessarily unique. If an arrow $\alpha$ is not contained in an infinite maximal path of $A$, there exists $n_\alpha$ such that $A\alpha A\subseteq A_{\leq n_\alpha}$ (that is, any path containing $\alpha$ has length at most $n_\alpha$).
    
    If $\alpha$ is contained in an infinite maximal path, then this infinite maximal path is $\gamma_\alpha\defeq (\gamma_r(\alpha))^\infty$ and any path in $\cB$ which contains $\alpha$ is a subpath of $\gamma_\alpha$.    
\end{corollary}

This allows us to partition $\cB_+$ into the subpaths of each infinite maximal path and the paths which are not subpaths of any infinite maximal path.

\begin{lemma}\label{lem:Abar}
Let $A=\kk\cQ/\cI$ be (locally) string. Let $\gamma_1,\ldots, \gamma_m$ be an enumeration of the infinite maximal paths of $A$. We define subquivers $\cQ^{(i)}$ of $\cQ$ as follows.
\begin{enumerate}[(i)]
    \item Let $\cQ_1^{(0)}=\setst{\alpha}{\alpha\not\leq \gamma_i\text{ for any }1\leq i\leq m},$ the set of arrows which are not contained in an infinite maximal path;
    \item For $1\leq i\leq m$, let $\cQ_1^{(i)}=\setst{\alpha}{\alpha\leq \gamma_i}$, the set of arrows  contained in $\gamma_i$.
    \item For $0\leq i\leq m$, let $\cQ_0^{(i)}=s(\cQ_1^{(i)})\cup t(\cQ_1^{(i)})$ (i.e., the vertices which are the source or target of some arrow in $\cQ_1^{(i)}$).
\end{enumerate} 
Let $\bar{A}\defeq \kk\cQ^{(0)}/(\cI\cap \kk\cQ^{(0)})$ and $A^{(i)}\defeq \kk\cQ^{(i)}/(\cI\cap \kk\cQ^{(i)})$ for $i>0$. Then, $\bar{A}_+$ and $A^{(i)}_+$ are ideals of $A$ and $A_+=\bar{A}_+\oplus \bigoplus_{i=1}^nA^{(i)}_+$.
\end{lemma}
The subalgebra $\bar{A}$ of $A$ is a direct sum of string algebras, though $\cQ^{(0)}$ is not necessarily connected. By contrast, each $A^{(i)}$ is locally gentle and has a unique maximal path ($\gamma_i$).

\begin{proof}
   Let $\alpha$ be any arrow in $\gamma_i$. For any $w,w'\in\cB$, either $w\alpha w'=0$ or $w\alpha w'$ is a path containing $\alpha$ and hence a subpath of $\gamma_i$ by Corollary \ref{cor:mpstr}. Thus
$A^{(i)}_+$ is an ideal of $A$ spanned over $\kk$ by the nontrivial subpaths of $\gamma_i$.

    Suppose that $w\in \cB_+$ is not a subpath of any infinite maximal path. Then none of the arrows in $w$ are contained in an infinite maximal path, so $w\in \bar{A}_+$. Additionally, for any $w'\in\cB_+$ such that $w'$ is a subpath of an infinite maximal path, we see that $ww'=w'w=0$ (otherwise, the resulting path would necessarily be a subpath of an infinite maximal path and hence $w$ would be as well). Thus $\bar{A}_+$ is also an ideal of $A$ and $A_+=\bar{A}_+\oplus \bigoplus_{i=1}^n A^{(i)}_+$.
\end{proof}

\begin{proposition}\label{prop:Jrad}
Let $\bar{A}$ be as in Lemma \ref{lem:Abar} and $\rJ(A)$ ($\grJ(A)$) denote the (graded) Jacobson radical of $A$. Then, $\grJ(\bar{A})=\bar{A}_+=\rJ(A)$.
\end{proposition}
\begin{proof}
By Corollary \ref{cor:mpstr}, there exists $n_\alpha$ such that any path containing $\alpha$ has length at most $n_\alpha$ for each $\alpha\in \cQ_1^{(0)}$. Taking $N=\max\setst{n_\alpha}{\alpha\in\cQ_1^{(0)}}$, we have $\bar{A}_+\subseteq A_{\leq N}$. Thus $\bar{A}_+$ is a nilpotent ideal of $A$ and so $\bar{A}_+\subseteq \rJ(A)$.

    Assume for contradiction that $x\in \rJ(A)\setminus \bar{A}_+$. Then there is an infinite maximal path $\gamma$ and nonstationary $w\leq \gamma$ such that $w$ appears with nonzero coefficient in $x$. As $\gamma$ is infinite, there exists $\alpha\in \cQ_1$ such that $\alpha w\notin\cI$. Then, there exists $w'\in\cB_+$ such that $\alpha w w'=(\gamma_r(\alpha))^k$ for some $k>0$.
    Then, $\alpha xw'=p(\gamma_r(\alpha))$ for some nonzero polynomial $p(t)\in t\kk[t]$ (as $F(\gamma_r(\alpha)) A L(\gamma_r(\alpha))\subseteq \kk[\gamma_r(\alpha)]$). As $x\in \rJ(A)$, we have $\one+p(\gamma_r(\alpha))\in A^\times$. Consider $(\one+p(\gamma_r(\alpha)))(\one+p(\gamma_r(\alpha)))^{-1}$. Left multiplication by $p(\gamma_r(a))$ will annihilate any path which is not an initial subpath of some power of  $\gamma_r(\alpha)$.
    Letting $w_j$ denote the initial subpath of length $j$ of $\gamma_r(\alpha)$ (so $w_{\ell(\gamma_r(\alpha))}=\gamma_r(\alpha)$),
    there exist $q_j(t)\in \kk[t]$ such that
    \begin{equation*}
        y\defeq (\one+p(\gamma_r(\alpha)))^{-1}-\left(\sum_{j=0}^{\ell(\gamma_r(\alpha))-1}w_jq_j(\gamma_r(\alpha))\right),
    \end{equation*}
    does not contain any initial subpath of $(\gamma_r(\alpha))^k$ with nonzero coefficient. Then,
    \begin{equation*}
        \one =(\one+p(\gamma_r(\alpha)))(\one+p(\gamma_r(\alpha)))^{-1} = y+\sum_{j=0}^{\ell((\gamma_r(\alpha)))-1}w_j((1+p)q_j)(\gamma_r(\alpha)).
    \end{equation*}
    Since $y$ does not contain any initial subpath of $(\gamma_r(\alpha))^k$ with nonzero coefficient, we must have 
    \begin{equation*}
        \sum_{j=0}^{\ell((\gamma_r(\alpha)))-1}w_j((1+p)q_j)(\gamma_r(\alpha))=w_0.
    \end{equation*}
     As $\set{w_j}_{\substack{0\leq j<\ell(\gamma_r(\alpha))}}$ is linearly independent over $\kk[\gamma_r(\alpha)]$, we obtain
    \begin{equation*}
    (1+p(t))q_0(t)=1
    \end{equation*}
    by considering the coefficient of $w_0$.
    As $p(t)\in t\kk[t]\setminus\{0\}$, this is a contradiction.
    \end{proof}
    In Section \ref{sec:ump}, we will make use of the locally gentle algebras with unique maximal paths. It is useful to consider such algebras as $\kk[x]$-algebras, which is possible in light of the following element. 
\begin{defn}\label{def:mgamma}
    For an infinite maximal path $\gamma$, we define $m_\gamma$ to be the sum of all the corresponding generating cycles. That is, $m_\gamma= \sum_{\setst{\alpha\in\cQ_1}{\gamma_\alpha=\gamma}} \gamma_r(\alpha)$.
\end{defn}
This $m_\gamma$ is a central element of $A$ and $\kk[m_\gamma]$ is isomorphic to the polynomial ring in one variable. In a (locally) gentle algebra $A$, if $\alpha\in\cQ_1$ is such that $\alpha A\alpha\neq 0$, then $\gamma_\alpha$ is infinite and $\alpha A\alpha$ is spanned by $\{\alpha m_{\gamma_\alpha}^k\}_{k>0}$. As any (locally) string algebra is a quotient of a (locally) gentle algebra by paths, it follows that if $aAa\neq 0$ for $A$ (locally) string and $a\in\cQ_1$, then $\gamma_a$ is infinite in any corresponding (locally) gentle algebra and so there is a homogeneous, central element $z_\alpha$ such that $\alpha A\alpha$ is spanned by $\{\alpha z_\alpha^k\}_{k>0}$. This $z_\alpha$ is the sum of rotations of the cycle $\gamma_r(\alpha)$. If $\alpha$ is contained in an infinite maximal path in $A$, then $z_\alpha=m_{\gamma_\alpha}$.

It was shown in \cite{FOZ} that the center of a (locally) gentle algebra is generated as a $\kk$-algebra by the elements $m_\gamma$ and the finite maximal paths which are cycles. There may be other generators in the (locally) string case, but none may be of degree zero.
\begin{lemma}\cite{FOZ}\label{lem:deg0ctr} If $A$ is (locally) string, then $Z(A)_0=\kk\one$. 

\end{lemma}
\begin{proof}
    Suppose that $x\in Z(A)_0$. Then, $x=\sum_{i\in\cQ_0}\lambda_i\epsilon_i$ for some $\lambda_i\in \kk$.
    For $\alpha\in\cQ_0$, we have
    \[\lambda_{\epsilon_{s(\alpha)}}\alpha=x\alpha=\alpha x=\lambda_{\epsilon_{t(\alpha)}}\alpha,\]
    so $\lambda_{\epsilon_{s(\alpha)}}=\lambda_{\epsilon_{t(\alpha)}}$. That is, $\lambda_i=\lambda_j$ if there is an arrow connecting $i$ and $j$. As $\cQ$ is connected, it follows that $\lambda_i=\lambda_j$ for any $i,j\in\cQ_0$ and thus
$x \in \kk\one $.
\end{proof}

\section{Initial properties of string automorphisms}\label{sec:ppoaut}
In \cite{Bl98}, it was shown that each automorphism of a string algebra has a corresponding automorphism of the quiver which preserves the ideal of relations (i.e., a corresponding graded automorphism). 
In Lemma \ref{lem:aut0}, we provide an adjusted version of this result for (locally) string algebras and, in Lemma \ref{lem:autpart}, we show that if the corresponding graded automorphism is the identity then an arrow $\alpha$ will be sent into the ideal it generates, up to a finite maximal path.
Towards those ends, we provide the following observation.

\begin{lemma}\label{lem:gengraut} 
Suppose $\Lambda=\bigoplus_{i=0}^\infty \Lambda_i$ is a locally-finite $\NN$-graded algebra. If $f\in \Aut(\Lambda)$ satisfies $f(\Lambda_{n})\subseteq \Lambda_{\geq n}$ for each $n\in\NN$, then there is a graded automorphism $\sigma$ of $\Lambda$ such that $f(x)\in \sigma (x)+\Lambda_{>n}$ for any $x\in \Lambda_n$ and $n\in\NN$.
	\end{lemma}
 
 \begin{proof} Assume that $f\in \Aut(\Lambda)$ satisfies $f(\Lambda_{n})\subseteq \Lambda_{\geq n}$ for each $n\in\NN$. Then, we may defined a graded map $\sigma:A\rightarrow A$ so that $f(x)\in \sigma(x)+\Lambda_{>n}$ for any $x\in \Lambda_n$ and $n\in\NN$. For any $m,n\in\NN$ and $x\in \Lambda_m$, $y\in \Lambda_n$, the degree $mn$ portion of $f(xy)=f(x)f(y)$ is $\sigma(x)\sigma(y)$, so $\gr f$ is a graded endomorphism. We will use the fact that $\Lambda$ is locally finite to show that $\sigma\vert_{\Lambda_{\leq n}}:\Lambda_{\leq n}\rightarrow \Lambda_{\leq n}$ is bijective for each $n$ via induction. For each $k$, let $\{v^{(k)}_j: 1\leq j\leq {\dim_\kk(\Lambda_k)}\}$ be a $\kk$-basis for $\Lambda_k$. 
 Suppose that $\sigma\vert_{\Lambda_{< n}}$ is bijective (this is trivially true if $n=0$) and let $x\in \Lambda_{n}$. As $f$ is surjective, 
 \[x=\sum_{k\in\NN}\sum_j\lambda^{(k)}_jf(v^{(k)}_j).\]
 for some $\lambda^{(n)}_j\in\kk$. Let $m=\min\{k\in \NN:\exists j (\lambda^{(k)}_j\neq 0)\}$.  As $f(v^{(k)}_j)\in \sigma(v^{(k)}_j)+A_{> k}$ for each $k$, 
 \[\sum_{j=1}^{\dim_\kk(\Lambda_m)}\lambda^{(m)}_j\sigma(v^{(m)}_j)=\sigma\left(\sum_{j=1}^{\dim_\kk(\Lambda_m)}\lambda^{(m)}_jv^{(m)}_j\right)\]
must be the degree $m$ component of $x$. By our choice of $m$, there exists $j$ such that $\lambda^{(m)}_j\neq 0$ and thus $\sum_{j=1}^{\dim_\kk(\Lambda_m)}\lambda^{(m)}_jv^{(m)}_j\neq 0$.
Assume for contradiction that $m< n.$
  Then $\sigma\vert_{\Lambda_{< n}}$ is bijective, so $x$ has a nonzero degree $m$ component. This contradicts our assumption that $x\in \Lambda_{n}$, so it must be that $m\geq n$. Consequently, the degree $n$ component of $x$ is
$\sum_{j=1}^{\dim_\kk(\Lambda_n)}\lambda^{(n)}_j\sigma(v^{(n)}_j).$
  Since $x\in \Lambda_{n}$, 
  \[x=\sum_{j=1}^{\dim_\kk(\Lambda_n)}\lambda^{(n)}_j\sigma(v^{(n)}_j).\]
It follows that $\sigma\vert_{\Lambda_{\leq n}}$ is surjective and hence bijective (as $\Lambda_{\leq n}$ is finite-dimensional). As $\sigma\vert_{\Lambda_{\leq n}}$ is bijective for all $n$, $\sigma$ is a graded automorphism of $\Lambda$.
\end{proof}

The following lemma gives sufficient conditions on $\Lambda$ to guarantee that $f(\Lambda_{n})\subseteq \Lambda_{\geq n}$ for each $n\in\NN$ and $f\in \Aut(\Lambda)$.

\begin{lemma}\label{lem:Jautpres} 
    Suppose $\Lambda$ is a homogeneous quotient of a path algebra on a finite quiver, so that $\Lambda$ is graded by path length and locally finite. If every loop in the quiver is nilpotent in $\Lambda$ or if $\Lambda\not\cong \kk[x]$ is  (locally) string, then every $f\in \Aut(\Lambda)$ satisfies $f(\Lambda_{n})\subseteq \Lambda_{\geq n}$ for all $n\in\NN$. In particular, under these conditions, $\grJ \Lambda$ is an algebraic invariant of $\Lambda$.
\end{lemma}
\begin{proof}
    Since $\Lambda$ is a quotient of a path algebra on a finite quiver graded by path length, it is generated by the stationary paths and arrows. Certainly, any $f\in \Aut(\Lambda)$ satisfies $f(\Lambda_0)\subseteq \Lambda_{\geq 0}=\Lambda$, so it suffices to show that $f(\alpha)\subseteq \Lambda_+$ for any arrow $\alpha$. If $\alpha^n=0$ for some $n$, this is clear, as $f(\alpha)^n=0$ implies that the coefficient of any stationary path in $f(\alpha)$ must be 0. Since any non-loop arrow is necessarily nilpotent in $\Lambda$, we see that it suffices that every loop in the quiver is nilpotent in $\Lambda$.

    Suppose instead that $\Lambda=\kk\cQ/\cI\not\cong \kk[x]$ is locally string and that $\alpha\in\cQ_1$ is a loop such that $\alpha^n\notin \cI$ for any $n\in \NN$. Then $\alpha$ generates an infinite maximal path $\alpha^\infty$ and  $m_{\alpha^\infty}=\alpha\in Z(\Lambda)$. As $f$ is an isomorphism, we have $f(\alpha)\in Z(\Lambda)$ and hence $f(\alpha)=\lambda \one+x$ for some $x\in \Lambda_+$ by Lemma \ref{lem:deg0ctr}.  
 Since $\Lambda\not\cong \kk[x]$, there must exist an arrow $\beta\in\cQ_1$ such that $\alpha \beta\in\cI$. Thus $f(\alpha)f(\beta)=0$, implying that $xf(\beta)=-\lambda f(\beta).$ Since $f$ is an automorphism, $f(\beta)\neq 0$, so there exists a maximal $i$ such that $f(\beta)\in \Lambda_{\geq i}$. As $\lambda f(\beta) = -xf(\beta)\in \Lambda_{> i}$,  it must that $\lambda=0$. Thus, we see that $f(\alpha)\in \Lambda_+$ in this case as well.
\end{proof}

\begin{lemma}\label{lem:aut0} Suppose $A=\kk\cQ/\cI$ is (locally) string. For each automorphism $f\in \Aut(A)$, there is a corresponding graded automorphism $\gr f$ and
\[1\rightarrow \Aut_0(A)\hookrightarrow \Aut(A)\xrightarrow[f\mapsto \gr f]{} \grAut(A)\rightarrow 1\]
is a short exact sequence, where $\Aut_0(A)$ denotes the automorphisms of $A$ for which the associated graded automorphism is the identity. If $A\not\cong \kk[x]$, then $f(w)\in \gr f(w)+A_{>\ell(w)}$ for all $w\in\cB$.
	\end{lemma}
 
 \begin{proof} If $A=\kk[x]$, then an automorphism $f\in \Aut(A)$ is given by $\lambda\in\kk$ and $\gamma\in \kk^\times$ where $f(x)=\gamma x+\lambda$, so we  define $\gr f$ by $\gr f(x)=\gamma x$. In this case, $\Aut_0(A)=\{x\mapsto x+\lambda\}_{\lambda\in\kk}$.

 If $A\not\cong \kk[x]$, we saw in Lemma \ref{lem:Jautpres} that for any  $f\in \Aut(A)$, we have $f(\Lambda_{n})\subseteq \Lambda_{\geq n}$ for each $n\in\NN$. By Lemma \ref{lem:gengraut}, there exists  $\gr f\in \grAut(A)$ such that $f(w)\in \gr f(w)+A_{>\ell(w)}$. 
\end{proof}

For a finite-dimensional split $\kk$-algebra $R$, \cite{GS99} defines certain subgroups ($H_A$, $\Hhat{A}$, $\Inn^*(A)$) of $\Aut(R)$ in terms of a fixed Wedderburn--Malcev decomposition $R=B\oplus \rJ(R)$. For a locally finite, $\NN$-graded algebra, we consider analogous subgroups defined in terms of the graded Jacobson radical. While the graded Jacobson radical is not an algebraic invariant in general, Lemma \ref{lem:Jautpres} shows that $\grJ(A)=A_+$ is an algebraic invariant of $A$ if $A\not\cong \kk[x]$ is (locally) string.

\begin{notation}
    Let $\Lambda$ be a locally finite, $\NN$-graded $\kk$-algebra and let $S\defeq \Lambda/\grJ(\Lambda)$. 
    For $x\in \Lambda^\times$, let $\Delta_x$ denote conjugation by $x$, so $\Delta_x(z)=x^{-1}zx$ and
    let $\Inn^*(\Lambda)\defeq \setst{\Delta_{\one+y}}{y\in \grJ(\Lambda), \one+y\in\Lambda^\times}$, the set of inner automorphisms of $\Lambda$ which correspond to an element of $\one+\grJ(\Lambda)$.

    If $\Lambda$ is elementary (i.e., $S\cong \kk^n$ for some $k$), then we may identify $S$ with a subspace of $\Lambda$ \cite{RR19}. In this case, 
    $\Hhat{\Lambda}\defeq \setst{\phi\in\Aut(\Lambda)}{\phi(S)\subseteq S}$ and 
    $H_\Lambda\defeq \setst{\phi\in\Aut(\Lambda)}{\phi\vert_S=\id_S}$.
\end{notation}
If $A$ is (locally) string,  we may identify $S$ with $\kk\cQ_0$. So $H_A$ is the subgroup of automorphisms which fix vertices, while $\Hhat{A}$ is the subgroup of automorphisms which permute vertices. Additionally, $\grJ(A)=A_+$, so $\Inn^*(\kk[x])=1$ and otherwise $\phi(A_+)\subseteq A_+$ for any automorphism $\phi$. Thus
$\Inn^*(A)$ is a normal subgroup of $\Aut_0(A)$  and 
    \[\rJ(A)\subseteq \setst{y\in \grJ(A)}{\one+y\in A^\times}\subseteq \grJ(A).\] 
These inclusions may both be strict, as we see in the following example.
\begin{example}
    Consider the locally gentle algebra $A=\kk\cQ$ where $\cQ$ is the quiver
    
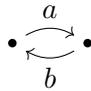
\begin{center}
        \begin{tikzpicture}
            \unvtx{1}{0,0}
            \unvtx{2}{1,0}
            \draw[->](1) to[bend left] node[midway, above]{$a$} (2);
            \draw[->](2) to[bend left] node[midway, below]{$b$} (1);
        \end{tikzpicture}.
    \end{center}
        Then $\one+a$ is a unit (with inverse $1-a$), but $\one+ab$ is not (as $ab$ is homogeneous but not nilpotent). 
       Therefore, $a\in \setst{y\in A_+}{\one+y\in A^\times}\setminus \rJ(A)$ and $ab\in\grJ(A)\setminus \setst{y\in A_+}{\one+y\in A^\times}$.
\end{example}

\begin{lemma}\label{lem:Inn}
Let $A$ be (locally) string. Then
 \begin{equation}\rJ(A)\subseteq \setst{y\in A_+}{\one+y\in A^\times}\subseteq \grJ(A).\end{equation}
 The first inclusion is strict  if and only if $A$ has an infinite maximal path containing more than one arrow. The second inclusion is strict if and only if $A$ has any infinite maximal path. 
\end{lemma}
\begin{proof}
    Let $\bar{A}$ be as in Lemma \ref{lem:Abar}. By Proposition \ref{prop:Jrad}, $\rJ(A)=\bar{A}_+$,  the span of the nontrivial subpaths of finite maximal paths. A string algebra has only finite maximal paths, so $\bar{A}=A$.
    
    Suppose $A$ is a locally string algebra with an infinite maximal path containing more than one arrow and let $\alpha$ be an arrow in that algebra. Then, $\alpha\neq \gamma_r(\alpha)$ is necessarily nilpotent and thus $\alpha\in\setst{y\in A_+}{\one+y\in A^\times}\setminus \rJ(A)$. However, $\gamma_r(\alpha)$ is homogeneous but not nilpotent and thus $\one+\gamma_r(\alpha)$ is not a unit. Consequently, $\gamma_r(\alpha)\in\grJ(A)\setminus \setst{y\in A_+}{\one+y\in A^\times}$.

    If $A$ is locally string algebra but none of its maximal paths contain more than one arrow, then let $\{\alpha_i\}_{i\in [n]}$ denote the arrows contained in infinite maximal paths. 
     Suppose that $y\in \setst{y\in A_+}{\Delta_{\one+y}\in \Inn^*(A)}$. As $y\in A_+$, we have $y=y'+\sum_{i=1}^n \alpha_ip_i(\alpha_i)$ for some $y'\in \bar{A}_+$ and polynomials $p_i\in \kk[x]$. By Lemma \ref{lem:Abar}, it follows that 
     \[\one+y=(\one+y')\prod_{i\in [n]} (\one+\alpha_ip_i(\alpha_i))\]
    and hence each $\one+\alpha_ip_i(\alpha_i)$ must also be a unit in $A$. The inverse of such an element must be of the same form (as $\alpha_iw=0$ if $w\in \cB_+\setminus \{\alpha_i^k\}$), but 
    \[(\one+\alpha_ip_i(\alpha_i))(\one+\alpha_iq_i(\alpha_i))=\one\]
    implies that $p_i=q_i=0$ (as $\alpha_i$ is not nilpotent). Hence $y=y'$ and $\setst{y\in A_+}{\one+y\in A^\times}=\bar{A}_+$.
\end{proof}

In particular, if $A$ is string, then our definition of $\Inn^*(A)$ matches that given in  \cite{GS99}. In that case, it is known that $\Aut(A)=\Hhat{A}\cdot \Inn^*(A)$. Note that if $\phi\in \Hhat{A}$, then  there exists $\sigma\in \grAut(A)$ such that $\sigma \circ \phi$ fixes vertices  by Lemma \ref{lem:aut0}. That is, $\Aut(A)=H_A\cdot \grAut(A)\cdot \Inn^*(A).$

\begin{lemma}\label{lem:autpart}
    Suppose that $A\not\cong \kk[x]$ is locally string and that $\phi\in \Aut_0(A)$. For $v\in\cQ_0$, we have $\phi(\epsilon_v)\in \langle \epsilon_v\rangle$. For $\alpha\in \cQ_1$, 
    we have $\phi(\alpha)\in y_\alpha+\langle \alpha\rangle$ where $y_\alpha\in A_{>1}$ is a linear combination of maximal paths parallel to $\alpha$. 
    That is, $\phi(\epsilon_v)$ is a linear combination of paths which pass through $v$, while $\phi(\alpha)$ is a linear combination of paths which contain $\alpha$ or are maximal and parallel to $\alpha$.
\end{lemma}
\begin{proof}
As $\phi\in \Aut_0(A)$, we have $\phi(\epsilon_v)\in x_{v}+\langle \epsilon_v\rangle$ for some $x_v\in A_+$ which is a linear combination of paths which do not pass through $\epsilon_v$. As $\epsilon_v$ is idempotent, 
\[\phi(\epsilon_v)-x_v^2=\phi(\epsilon_v)^2-x_v^2\in \langle \epsilon_v\rangle,\]
so $x_v-x_v^2\in \langle \epsilon_v\rangle$. Since $x_v-x_v^2$ is a linear combination of paths not passing through $\epsilon_v$, it follows that $x_v=x_v^2$ and hence, as $x_v\in A_+$,  $x_v=0$.

 Similarly, $\phi\in \Aut_0(A)$ implies $\phi(\alpha)\in y_{\alpha}+\langle \alpha\rangle$ for some $y_\alpha\in A_{>1}$ which is a linear combination of paths which do not contain $\alpha$.
Assume for contradiction that  $y_\alpha$ is not a linear combination of maximal paths. Then there is a non-maximal path $w\in A_{>1}\setminus\langle \alpha\rangle$ which appears with nonzero coefficient in $y_\alpha$. We may assume that $w$ is of minimal length among such paths. As $w$ is not maximal, there exists an arrow $\beta$ such that $\beta w\notin\cI$ or $w\beta\notin\cI$. The cases are similar, so assume the former. As $y_\alpha\in A_+\setminus \langle \alpha\rangle$, $\beta w\notin\cI$ implies $\beta\alpha\in \cI$. Thus, 
\[0=\phi(\beta)\phi(\alpha)\in (\beta+A_{>1})\phi(\alpha)\]
As the coefficient of $\beta w$ in $\beta\phi(\alpha)$ is nonzero, there must exist $w'\in \cB_{>1}$ and  $w''\in\cB$ such that $w'w''=\beta w$ and  $w''$ occurs in $\phi(\alpha)$ with nonzero coefficient. As $w'\in \cB_{>1}$, it follows that
$w''\lneq w$. Thus $w''\notin\langle \alpha\rangle$ is not maximal and appears with nonzero coefficient in $y_\alpha$, but $\ell(w')<\ell(w)$. This contradicts our minimality assumption, so $y_\alpha$ is a linear combination of maximal paths.
 
 To see that these maximal paths must be parallel to $\alpha$, consider $\phi(\epsilon_{s(\alpha)})\phi(\alpha)\phi(\epsilon_{t(\alpha)})$. As
 $\phi(\epsilon_v)\in \epsilon_v+A_+$ for any $v\in \cQ_0$ and $A_+y_\alpha=y_\alpha A_+=0$, we see that
\[\phi(\alpha)=\phi(\epsilon_{s(\alpha)})\phi(\alpha)\phi(\epsilon_{t(\alpha)})\in \epsilon_{s(\alpha)} y_\alpha\epsilon_{t(\alpha)}+\langle a\rangle. \]
As $\epsilon_{s(\alpha)} y_\alpha\epsilon_{t(\alpha)}$ is also a linear combination of paths in $\cB\setminus \langle a\rangle$, it follows that $\epsilon_{s(\alpha)} y_\alpha\epsilon_{t(\alpha)}=y_\alpha$.
\end{proof}

\section{Exponential automorphisms and the finite-dimensional case}\label{sec:exp}
In this section, we will consider the automorphisms of (locally) string algebra which arise from derivations, which are known as exponential automorphisms, and show that if $A$ is string, then any automorphism in $\Aut_0(A)$ may be decomposed into an inner automorphism and a composition of exponential automorphisms.

\subsection{Exponential automorphisms}
\begin{lemma}\label{lem:basicder} 
Let $A$ be a (locally) string algebra and $\cB$ be the basis of paths for $A$.  
For $a\in\cQ_1$ and $p\in \cB$, there is a derivation $\delta_{(a,p)}$ of $A$ which takes $a$ to $p$ and sends every other arrow and stationary path to 0 if the following hold:
\begin{enumerate}[(i)]
    \item $p$ is parallel to $a$,
    \item $p$ is left maximal or $F(p)=a$, and
    \item $p$ is right maximal or $L(p)=a$.
\end{enumerate}
	\end{lemma}

    While these conditions are necessary and sufficient for such a derivation to exist, we only show the latter here.    If $p_i\in \cB$ all satisfy  (i-iii) and $\lambda_i\in \kk$, we let $\delta_{(a,\sum \lambda_ip_i)}$ denote the derivation $\sum \lambda_i\delta_{(a,p_i)}$. 
\begin{proof}
    For any $p$, we may define a map $d_{(a,p)}:\cB\rightarrow A$ by letting $d(\epsilon_v)=0$ for any $v\in \cQ_0$;
    $d(a)=p$ and $d(\alpha)=0$ for any $\alpha\in \cQ_1\setminus \{a\}$; and, for $w\in\cB_{+}$,
    \[d(w)=
    \sum_{i=1}^n\alpha_1\cdots d(\alpha_i)\cdots \alpha_n\]
    where $w=\alpha_1\cdots \alpha_n$ and $\alpha_i\in \cQ_1$. Note that, if $w,w'\in\cB_+$ are such that $ww'\notin \cI$, we have
    \begin{equation} d_{(a,p)}(ww')=wd_{(a,p)}(w')+d_{(a,p)}(w)w'.\label{eq:derel}\end{equation}
    Then, we see that extending $d_{(a,p)}$ to a linear endomorphism of $A$ yields a derivation provided that (1) \eqref{eq:derel} also holds  if one (or both) of $w,w'$ are stationary and (2)   $wd_{(a,p)}(w')+d_{(a,p)}(w)w'=0$ if $ww'\in \cI$. 

    Note that \eqref{eq:derel} holds whenever one (or both) of $w,w'$ are stationary provided $d_{(a,p)}(w)\in \epsilon_{s(w)}A\epsilon_{t(w)}$ for all $w\in\cB$, which holds exactly when $p$ is parallel to $a$. This also implies that the second property holds if one or both of $w,w'$ is stationary.

    Suppose that $p$ satisfies the conditions (i-iii). Assume for contradiction that there exist $w,w'\in\cB$ such that $waw'\in \cI$ and $wpw'\notin\cI$. As $p$ is parallel to $a$, it is clear that $waw'\notin\cI$ if $w,w'$ are both stationary. Without loss of generality, we may suppose that $w$ is nonstationary. Then, $wp\notin\cI$ implies that $p$ is not left maximal and thus $F(p)=a$. As $wa\leq wp$, it follows that $wa\notin\cI$. Hence $w'$ is nonstationary as well, so $L(p)=a$. Certainly $p\neq a$, so we have $p\in a Aa$. As noted after Definition \ref{def:mgamma}, $aAa\neq 0$ implies that there exists a central element $z$ (the sum of all primitive cycles containing $a$) such that $aAa$ is spanned by $\{az^k\}_{k>0}$. In particular, $p=az^k$ for some $k$. However, this implies that $wpw'=waw'z^k\in \cI$ as well, a contradiction. Thus, if $p$ satisfies the conditions (i-iii), then $waw'\in \cI$ implies $wpw'\in\cI$.

    Suppose that $w,w'\in\cB_+$ are such that $ww'\in\cI$. Then there exist arrows $\alpha_i\in\cB$ such that $w= \alpha_1\cdots\alpha_n$ and $w'=\alpha_{n+1}\cdots \alpha_m$. By definition of $d_{(a,p)}$, we have 
    note that 
    \[wd_{(a,p)}(w')+d_{(a,p)}(w)w'=\sum_{i=1}^m \alpha_1\cdots \alpha_{i-1}d(\alpha_i)\alpha_{i+1}\cdots \alpha_m=\sum_{\alpha_i=a} \alpha_1\cdots \alpha_{i-1}p\alpha_{i+1}\cdots \alpha_m.\]
    By assumption, if $\alpha_i=a$, then 
    \[\alpha_1\cdots \alpha_{i-1}a\alpha_{i+1}\cdots \alpha_m\in\cI.\] Thus $wd_{(a,p)}(w')+d_{(a,p)}(w)w'\in\cI$ if $p$ satisfies conditions (i-iii), as desired.
    \end{proof}

\begin{defn}
    A derivation $\delta$ of $A$ is called \emph{locally nilpotent} if, for each $x\in A$, there exists $n>0$ such that $\delta^n(x)=0$.    For such a derivation, there is an automorphism $\exp(\delta)$ of $A$ given by
    \begin{align*}
        \exp(\delta)(x)=\sum_{n\geq 0}\frac{1}{n!} \delta^n(x).
    \end{align*}
    Automorphisms of this form are called \emph{exponential automorphisms} of $A$.
\end{defn}

\begin{example}
 Let $a\in\cQ_1$. Then $\delta_{(a,a)}$ is never locally nilpotent, as $\delta_{(a,a)}^n(a)=a$ for all $n>0$. \\
 \noindent For   a linear combination $y\in \epsilon_{s(a)}A_{>1}\epsilon_{t(a)}$ of maximal paths,  $\delta_{(a,y)}$ is a nilpotent derivation of $A$.
 For  $x_a\in aAa$, the derivation $\delta_{(a,x_a)}$ is locally nilpotent if $a$ is contained in a finite maximal path. 
\end{example} 

There are two types of nilpotent derivations on $A$ on which we will focus.
\begin{defn}
Let ${\cQ}^{(0)}_1$ denote the set of arrows which are not contained in any infinite maximal path (as in  Lemma \ref{lem:Abar}).
A derivation $\delta$ is  of \emph{type I} if $\delta=\sum_{a\in {\cQ}^{(0)}_1} \delta_{(a,x_a)}$ where $x_a\in aAa$.

    A derivation $\delta$ is of \emph{type II} if $\delta=\sum_{a\in {\cQ}_1} \delta_{(a,y_a)}$ where, for each $a\in\cQ_1$, $y_a\in A_{>1}$ is a linear combination of left maximal paths satisfying conditions (i-iii) of Lemma \ref{lem:basicder}.

    We say that an exponential automorphism $\exp(\delta)$ is of type I (resp. II) if $\delta$ is of type I (resp. II).
\end{defn}

\begin{lemma}\label{lem:nilpder}
    Let $A=\kk\cQ/\cI$ be a (locally) string algebra and $\cB$ the basis of paths for $A$. 

    Derivations of type I and type II are nilpotent and the corresponding exponential automorphisms belong to $H_A\cap \Aut_0(A)$. Moreover, derivations of type II commute with derivations of type I and type II. In particular, if $\alpha\in \cQ_1$ and $y\in \epsilon_{s(\alpha)}A_{>1}\epsilon_{t(\alpha)}$ is a linear combination of left maximal paths which are also right maximal or have last arrow $\alpha$,  then
    \[\delta_{(\alpha,y)}\circ \delta=\delta\circ \delta_{(\alpha,y)}=0\]
    for a derivation $\delta$ of type I or II.
\end{lemma}
\begin{proof}
Suppose that $\delta$ is of type I. Then, there exist $\alpha_1,\ldots, \alpha_n\in \cQ^{(0)}_1$ and $x_i\in \alpha_iA\alpha_i\setminus\{0\}$ such that $\delta =\sum_{i=1}^n\delta_{(\alpha_i,x_i)}.$ Since $\alpha_iA\alpha_i\neq 0$, it is spanned by $\{\alpha_iz_i^k\}_{k> 0}$ for some $z_i\in Z(A)$ (see Section \ref{sec:prelim}). Thus, $x_i=\alpha_ip_i(z_i)$ for some polynomial $p_i$ with coefficients in $\kk$ satisfying $p_i(0)=0$. For $w\in \cB$, if we let  $n_i(w)$ denote the number of copies of $\alpha_i$ in $w$, we see that
    \[\delta_{(\alpha_i,\alpha_iz_i^j)}(w)=n_i(w)z^{j}w\]
   because $z_i$ is central. Consequently, 
    \begin{equation}\delta(w)=\sum_{i=1}^n n_i(w)p_i(z_i)w.\label{eq:type1}\end{equation}
Since $\alpha_i\in \cQ^{(0)}_1$, we have $p_i(z_i)\in \bar{A}_+$. As $\bar{A}_+$ is an ideal of $A$, it follows that $\delta(A)\subseteq \bar{A}_+$. Additionally, as $p_i(z_i)\in A_+$, we have $\delta(A_i)\subseteq A_{>i}$ and hence
$\delta^n(A)\subseteq \bar{A}_{>n}.$
However, as noted in Proposition \ref{prop:Jrad}, there exists $N$ such that $\bar{A}_{> N}=0$, so $\delta$ is nilpotent.

Suppose that $\alpha\in \cQ_1$ and $y\in \epsilon_{s(\alpha)}A_{>1}\epsilon_{t(\alpha)}$ is a linear combination of left maximal paths which are also right maximal or have last arrow $\alpha$. That is, $y=y_1+y_2a$ where $y_1\in A_{>1}$ is a linear combination of maximal paths parallel to $\alpha$ and $y_2\in A_+$ is such that $y_2\alpha$ a linear combination left maximal paths. 
Then, for any $w\in\cB_+$,
    \[\delta_{(\alpha,y)}(w)=\begin{cases} 0 & \text{ if $w\in\cB_0$,}\\y & \text{ if $w=\alpha$},\\
    y_2w & \text{ otherwise.}\end{cases}\]
    Hence $\delta_{(\alpha,y)}(w)$ is a linear combination of left maximal paths in $A_{>1}$ for every $w\in \cB$. Consequently, as equation \eqref{eq:type1} shows that $\delta$  kills left maximal paths, we see that $\delta\circ \delta_{(\alpha,y)}=0$. Additionally, since $\delta(A)\subseteq A_{>1}$ and $z_i$ are central, we see that 
    \[(\delta_{(\alpha,y)}\circ \delta)(w)=y_2\sum_{i=1}^n n_i(w)p_i(z_i)w=\sum_{i=1}^n n_i(w)p_i(z_i)y_2w.\]
    If $y_2w$ is neither zero nor left maximal, it must be that $w$ is stationary and $n_i(w)=0$ for all $i$. Thus, $\delta_{(\alpha,y)}\circ \delta=0$ as well. 
    
    If $\alpha'\in \cQ_1$ and  $y'\in \epsilon_{s(\alpha')}A_{>1}\epsilon_{t(\alpha')}$ is a linear combination of left maximal paths which are also right maximal or have last arrow $\alpha'$, then $(\delta_{(\alpha,y)}\circ \delta_{(\alpha',y')})(w)=y_2\delta_{(\alpha',y')})(w)=0$ for any $w\in\cB$, so
    \[\delta_{(\alpha,y)}\circ \delta_{(\alpha',y')}=0.\]
    Since derivations of type II are sums of such derivations, it follows they are nilpotent of degree two and commute with derivations of type I and II. As derivations of both types kill stationary paths and take $\cQ_1$ to $A_{>1}$, the corresponding exponential automorphisms belong to $H_A\cap \Aut_0(A)$.
    
\end{proof}

In particular, for any choice of arrows $\alpha_i$ and linear combinations $y_i\in \epsilon_{s(\alpha_i)}A_{>1}\epsilon_{t(\alpha_i)}$ of maximal paths, there is an automorphism which takes $\alpha_i$ to $\alpha_i+y_i$ and fixes all other paths.
\begin{example} 
Suppose $\delta = \sum_{i=1}^n\delta_{(\alpha_i,y_i)}$ where $\alpha_i\in\cQ_1$ and each $y_i\in \epsilon_{s(\alpha_i)}A_{>1}\epsilon_{t(\alpha_i)}$ is a linear combination of maximal paths. Then for any $w\in \cB$,
\[\delta(w)=\begin{cases}
    y_i & \text{ if $w=\alpha_i$ for some $i$},\\
    0 & \text{ otherwise}.
\end{cases}\]
Since $\delta^2=0$, it follows that $\exp(\delta)$ sends $\alpha_i$ to $\alpha_i+y_i$ and fixes all other elements of $\cB$.
\end{example}
This gives us the following corollary to Lemma \ref{lem:autpart}.

\begin{corollary}\label{cor:fixAbar} Let $A\not\cong \kk[x]$ be a (locally) string algebra.  Suppose that $\phi\in \Aut_0(A)$. For $\alpha\in \cQ_1$, let $y_\alpha$ be as in Lemma \ref{lem:autpart}. Then $\delta\defeq -\sum_{\alpha\in \cQ_1}\delta_{\alpha,y_\alpha}$ is a derivation of type II and $(\exp(\delta)\circ\phi)(\alpha)\in \langle \alpha\rangle$ for every $\alpha\in \cQ_1$.
\end{corollary}

\subsection{String automorphisms}
In this section, we consider the case where $A$ is finite-dimensional. In this case, $\cQ^{(0)}=\cQ$ and $\bar{A}=A$.

\begin{lemma}\label{lem:p(z)type}
    Suppose $A$ is a string algebra. Let $D(A)$ denote the set of automorphisms $\phi\in H_A\cap \Aut_0(A)$ such that $\phi(w)\in \kk\setst{w'\in \cB}{F(w')=F(w), L(w')=L(w)}$ for each $w\in\cB$.
    Then $D(A)$ is the subgroup of $\Aut_0(A)$ generated by exponential automorphisms of type I. Moreover, for any choice of $x_\alpha\in \alpha A\alpha$, there exists $\rho\in D(A)$ such that $\rho(\alpha)=\alpha+x_\alpha$ for each arrow $\alpha$.
    
\end{lemma}
If $A$ is gentle, then $\alpha A\alpha =0$ for every arrow $\alpha$ and so $D(A)$ is trivial.

\begin{proof}
Suppose that $\delta$ is a derivation of type I. Then, as seen in Lemma \ref{lem:nilpder}, $\delta$ is nilpotent and there exist central elements $p_\alpha$ so that for any $w\in \cB$ we have
\[\delta(w)=\sum_{\alpha\in\cQ_1} n_\alpha (w)p_\alpha w,\]
where $n_\alpha(w)$ is the number of copies of $\alpha$ in $w$. It follows that 
\[\delta(w)\in \kk\setst{w'\in \cB}{F(w')=F(w), L(w')=L(w)}\]
and hence $\exp(\delta)\in D(A)$. Thus $D(A)$ contains type I exponential automorphisms.

Since $A$ is generated in degree $\leq 1$ and $\phi\in D(A)$ must fix vertices and take each $\alpha\in \cQ_1$ to $\alpha+\alpha A\alpha$, it suffices to show that for any choice of $x_\alpha\in \alpha A\alpha$, there is some composition $\rho$ of type I exponential automorphism such that $\rho(\alpha)=\alpha+x_\alpha$ for each $\alpha\in\cQ_1$.
As $A$ is finite-dimensional, there exists $N$ such that $A_{\geq N}=0$. Hence, for any choice of $x_\alpha\in \alpha A\alpha$, there is a minimal $n$ such that $x_\alpha \in A_{\geq N-n}$ for all $\alpha\in\cQ_1$. We proceed via induction on $n$. If $n=0$, then $x_\alpha=0$ for every arrow $\alpha$ and we may take $\rho=\id_A$.

Suppose $n>0$, and that such a $\rho$ exists for any choice of elements in $\alpha A\alpha\cap A_{\geq N-(n-1)}$. Let $x_\alpha'$ denote the degree $N-n$ portion of $x_\alpha$ for each $\alpha\in\cQ_2$, so that $x_\alpha-x_\alpha'\in A_{> N-n}$. 
Note that $x_\alpha'\in \alpha A\alpha$ as well, so $\delta\defeq \sum_{\alpha\in\cQ_1} \delta_{(\alpha, x_\alpha')}$ is of type I. Since $\delta(\alpha)=x_\alpha'$ and $\delta(A_i)\subseteq A_{>i}$ for any $i$, we have
$\exp(-\delta)(\alpha)\in \alpha-x_\alpha'+A_{>N-n}$ and $\exp(-\delta)(x_\alpha')\in x_\alpha'+A_{>N-n}$. Thus
\[\exp(-\delta)(\alpha+x_\alpha)\in \alpha+A_{>N-n}.\]
As $-\delta$ is type I, $\exp(-\delta)\in D(A)$ and so $\exp(-\delta)(\alpha+x_\alpha)\in \alpha+\alpha A\alpha$. Hence $\exp(-\delta)(\alpha+x_\alpha)-\alpha\in \alpha A\alpha\cap A_{\geq N-(n-1)}$, so there exists a composition $\rho$ of type I exponential automorphisms such that 
$\rho(\alpha)=\exp(-\delta)(\alpha+x_\alpha)$. Then, $\exp(\delta)\circ \rho$ is a composition of type I exponential automorphisms and $(\exp(\delta)\circ \rho)(\alpha)=\alpha+x_\alpha$, as desired. 
\end{proof}

The automorphisms described in Lemma \ref{lem:p(z)type} may be inner or outer.
\begin{example}\label{ex:innerD(a)}
    Consider $A=\kk\cQ/\langle (ab)^2,(ba)^2\rangle$ where again $\cQ$ is the quiver 
    
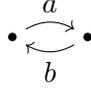
\begin{center}
        \begin{tikzpicture}
            \unvtx{1}{0,0}
            \unvtx{2}{1,0}
            \draw[->](1) to[bend left] node[midway, above]{$a$} (2);
            \draw[->](2) to[bend left] node[midway, below]{$b$} (1);
        \end{tikzpicture}.
    \end{center}
Then $\begin{cases} a\mapsto a+aba,\\
w\mapsto w\text{ if }w\in \cB\setminus\{a\}\end{cases}$ defines an outer automorphism of $A$, while 
\[\begin{cases} 
a\mapsto a+2aba,\\
b\mapsto b-2bab,\\
w\mapsto w  \text{ if }w\in \cB\setminus\{a,b\}\end{cases}\]
defines an inner automorphism of $A$ ($\Delta_{\one-ab+ba}$).

\end{example}

We will employ a similar argument for $H_A\cap \Aut_0(A)$. To that end, we have the following lemma.
\begin{lemma}\label{lem:sgrfrm}
    Let $A$ be (locally) string and suppose $\phi\in H_A\cap \Aut_0(A)$.  There exists $n>1$ and $\lambda_{(\alpha, w)}\in\kk$ such that
    \[\phi(\alpha)=\alpha+\sum_{w\in B_{\geq n}}\lambda_{(\alpha, w)}w\]
     for each arrow $\alpha$. If $w\in B_{n}$ and $\lambda_{(\alpha, w)}\neq 0$, then
     \begin{enumerate}[(i)]
         \item $w$ is parallel to $\alpha$.
         \item If $F(w)\neq \alpha$ and $w$ is not left maximal, then there exists $w'$ such that $w=w'\alpha$ and $\lambda_{(\alpha,w)}=-\lambda_{(\beta, \beta w')}$ for some arrow $\beta$.
         \item if $F(w), L(w)\neq \alpha$, then $w$ is maximal.
     \end{enumerate}
\end{lemma}
\begin{proof}
    Suppose $\phi\in H_A\cap \Aut_0(A)$. 
    
    (i) Since $\phi\in H_A$ fixes stationary paths, $\phi(\alpha)=\epsilon_{s(\alpha)}\phi(\alpha)\epsilon_{t(\alpha)}$ is a linear combination of paths parallel to $\alpha$. Thus $\lambda_{(\alpha, w)}\neq 0$ implies $w$ is parallel to $\alpha$.

 (ii) Suppose that $w\in \cB_{n(\phi)}$ is such that $\lambda_{(\alpha, w)}\neq 0$ but $F(w)\neq \alpha$ and $w$ is not left maximal.
Then, there exists an arrow $\beta$ such that $\beta w\notin \cI$ but $\beta\alpha\in \cI$, so
\[0=\phi(\beta)\phi(\alpha)=\left(\beta+\sum_{\ell(w)\geq n(\phi)}\lambda_{(\beta, w)}w\right)\left(\alpha+\sum_{\ell(w)\geq n(\phi)}\lambda_{(\alpha, w)}w\right).\]
Since $\beta w\notin\cI$, the coefficient of $\beta w$ in this sum must be 0. Since $w\in \cB_{n(\phi)}$, there must exist $w'\in\cB$ such that $w=w'\alpha$ and $\lambda_{(\beta, \beta w')}=-\lambda_{(\alpha,w)}$. In particular, $L(w)=\alpha$.

(iii) Similarly, if $w\in \cB_{n(\phi)}$ is such that $\lambda_{(\alpha, w)}\neq 0$ but $w$ is not right maximal and $L(w)\neq \alpha$, then $F(w)=\alpha$. Consequently, if $F(w)\neq \alpha$ and $L(w)\neq \alpha$, then $w$ must be maximal.
\end{proof}

\begin{proposition}\label{prop:strfixv}
    Suppose that $A$ is a string algebra and $\phi\in H_A\cap \Aut_0(A)$. Then, there exist $y\in A_+$, $\rho\in D(A)$, and a type II derivation $\delta$ such that  $\phi =  \exp(\delta)\circ \rho\circ \Delta_{\one+y}.$
\end{proposition}
\begin{proof}
Suppose $\phi\in H_A\cap \Aut_0(A)$. Then, there exists $n>1$ such that  $\phi(\alpha)-\alpha \in A_{\geq n}$ for each arrow $\alpha$. Since $A$ is finite-dimensional, there exists $N$ such that $A_{\geq N}=0$. Thus, $1<n\leq N$. We proceed via induction on $N-n$. If $N-n=0$, then $\phi(\alpha)=\alpha$ for every arrow $\alpha$. Since $\phi\in H_A$ also fixes the stationary paths, it must be that $\phi=\id_A$.

Assume that for any $\psi\in \Aut_0(A)$ with $\psi(\alpha)-\alpha \in A_{>n}$ for each arrow $\alpha$, $\psi$ decomposes into an element of $\Inn^*(A)$, and element of $D(A)$, and a type II exponential automorphism. Then
\[\phi(\alpha)=\alpha+\sum_{\ell(w)\geq n}\lambda_{(\alpha, w)}w.\]
for some  $\lambda_{(\alpha,w)}\in\kk$. Define 
\[y=\sum_{w\in \cB_{n-1}}\sum_{\beta\in\cQ_1}\lambda_{(\beta,\beta w)}w.\]
Since $\lambda_{(\beta,\beta w)}\neq 0$ implies $\beta w$ is parallel to $\beta$ (hence $w$ is a cycle at $t(\beta)$),  $y$ must be a linear combination of cycles in $A_{n-1}\subseteq A_+$. Since $A_{\geq N}$ is finite-dimensional, this means that $\one-y\in A^\times$ and $(\one-y)^{-1}=\one+y+y^2+\cdots+y^N$. Thus, $\Delta_{\one-y}(w)\in \Inn^*(A)\subseteq \Aut_0(A)$.
As $y$ is sum of cycles,  $\Delta_{\one-y}\in H_A\cap \Aut_0(A)$. For an arrow $\alpha$,
\[\Delta_{\one-y}(\alpha)\in \alpha+\sum_{w\in \cB_{n-1}}\sum_{\beta\in\cQ_1}\lambda_{(\beta,\beta w)}w\alpha-\sum_{w\in \cB_{n-1}}\lambda_{(\alpha,\alpha w)}\alpha w+A_{>n}\]
and so
$\Delta_{\one-y}(\phi(\alpha))\in \alpha+\sum\limits_{w\in \cB_{n-1}}\sum\limits_{\beta\in\cQ_1}\lambda_{(\beta,\beta w)}w\alpha+\sum\limits_{\substack{w'\in \cB_{n}\\ F(w')\neq \alpha}}\lambda_{(\alpha, w')} w'+A_{>n}.$ Let
\[x_\alpha\defeq \sum_{\substack{w\in \cB_{n(\phi)-1}\\F(w)=\alpha}}\sum_{\beta\in\cQ_1}\lambda_{(\beta,\beta w)}w\alpha\in \alpha A \alpha.\]
Then $\sum_{\alpha\in\cQ_1}\delta_{(\alpha,x_\alpha)}$ is of type I and so $\rho:=\exp\left(\sum_{\alpha\in\cQ_1}\delta_{(\alpha,x_\alpha)}\right)\in D(A)$. Hence
\[(\rho\circ \Delta_{\one-y})(\phi(\alpha))\in \alpha+\sum_{\substack{w\in \cB_{n-1}\\F(w)\neq \alpha}}\sum_{\beta\in\cQ_1}\lambda_{(\beta,\beta w)}w\alpha+\sum_{\substack{w'\in \cB_{n}\\ F(w')\neq \alpha}}\lambda_{(\alpha, w')} w'+A_{>n}\]
for each arrow $\alpha$. That is, 
\[(\rho\circ \Delta_{\one-y}\circ \phi)(\alpha)\in \alpha+\sum_{\substack{w\in \cB_{\geq n}}}\bar{\lambda}_{(\alpha, w)} w\]
for  $\bar{\lambda}_{(\alpha, w)}\in \kk$ such that $\bar{\lambda}_{(\alpha, w)}=0$ for $w\in \cB_n$ with $F(w)=\alpha$. Since $\rho\in D(A)$ and $\Delta_{\one-y}\in H_A\cap \Aut_0(A)$, it follows that $\rho\circ \Delta_{\one-y}\circ \phi\in H_A\cap \Aut_0(A)$. Suppose that there exists an arrow $\alpha$ and path $w\in \cB_{n}$ such that $\bar{\lambda}_{(\alpha, w)}\neq 0$, but $w$ is not left maximal. By Lemma \ref{lem:sgrfrm}, there must exist $w'\in \cB_{n-1}$ and $\beta\in \cQ_1$ such that $w=w'\alpha$ and $\bar{\lambda}_{(\alpha, w)}=-\bar{\lambda}_{(\beta, \beta w')}$. However, $\bar{\lambda}_{(\beta, \beta w')}=0$ for every arrow $\beta$. Thus, every $w\in \cB_{n}$ with $\bar{\lambda}_{(\alpha, w)}\neq 0$ must be left maximal. Lemma \ref{lem:sgrfrm} also tells us that $w$ is parallel to $\alpha$ and right maximal if $L(w)\neq \alpha$. Thus, $\bar{y}_\alpha\defeq \sum_{w\in \cB_{n}}\bar{\lambda}_{(\alpha, w)}w$ is a linear combination of left maximal paths parallel to $\alpha$, which are either right maximal or end with $\alpha$, so 
\[\delta:=\sum_{\alpha\in \cQ_1}\delta_{\alpha, -\bar{y}_\alpha}\]
is a type II derivation and  $\exp(\delta)\in H_A\cap \Aut_0(A)$. Consider 
\begin{equation}\hat{\phi}:=\exp(\delta)\circ \rho\circ \Delta_{\one-y}\circ \phi\in H_A\cap \Aut_0(A) .\end{equation}
By choice of $\bar{y}_\alpha$,
we have
$\hat{\phi}(\alpha)\in \alpha+ A_{>n}$ for each arrow $\alpha$.
Hence there exists $y'\in A_+$, $\rho'\in D(A)$, and a type II derivation $\delta'$ such that  
$\hat{\phi} =  \exp(\delta')\circ \rho'\circ \Delta_{\one+y'}.$
Thus 
\[\phi =\left(\exp(\delta)\circ \rho\circ \Delta_{\one-y}\right)^{-1}\circ \hat{\phi}=\left(\Delta_{(\one-y)^{-1}}\circ \rho^{-1}\circ \exp(-\delta)\right)\circ \left( \exp(\delta')\circ \rho'\circ \Delta_{\one+y'}\right).\]
By Lemma \ref{lem:nilpder}, $\delta$ and $\delta'$  commute and hence $\exp(-\delta)\circ \exp(\delta')=\exp(\delta'-\delta)$. Since $\rho\in D(A)$, we know that $\rho^{-1}$ is also a composition of type I exponential automorphisms and hence commutes with $\exp(\delta'-\delta)$ by Lemma \ref{lem:nilpder}, so
\[\phi =\Delta_{(\one-y)^{-1}}\circ \exp(\delta'-\delta)\circ (\rho^{-1}\circ \rho')\circ \Delta_{\one-y'},\]
where $\rho' \circ \rho^{-1}\in D(A)$ and $\delta'-\delta$ is of type II. As $\Inn^*(A)$ is a normal subgroup of $\Aut(A)$, the desired result follows.
\end{proof}
As $\Aut(A)=\Hhat{A}\cdot \Inn^*(A)$ \cite{GS99, Po89}, we may describe $\Aut_0(A)$ as follows.
\begin{theorem}\label{thm:strA0}
    Let $A$ be a string algebra. Then, up to an inner automorphism, any $\phi\in \Aut_0(A)$ may be decomposed into exponential automorphisms of type I and II.
\end{theorem}
\begin{proof}
    Suppose that $\phi\in \Aut_0(A)$. Then,  $\phi=\hat{\phi}\circ \Delta_{\one+y}$ for some $y\in A_+$ and $\hat{\phi}\in \Hhat{A}$. Since $y\in A_+$, we have $\Delta_{\one+y}\in \Aut_0(A)$ and hence $\hat{\phi}\in \Aut_0(A)$. As $\Aut_0(A)\cap \Hhat{A}=\Aut_0(A)\cap H_A$, Proposition \ref{prop:strfixv} tells us that $\hat{\phi}=\exp(\delta)\circ \rho\circ \Delta_{\one+y'}$ where $\rho\in D(A)$ and $\delta$ is a type II derivation. Hence
    \[\phi=\exp(\delta)\circ \rho\circ \Delta_{(\one+y)(\one+y')},\]
    is the desired decomposition. 
\end{proof}

\section{Automorphisms of a locally gentle with unique maximal path}\label{sec:ump}
In this section, we consider automorphisms of locally gentle algebras with a unique maximal path by embedding them in $M_n(\kk[x])$. In Proposition \ref{prop:1mcase}, we conclude that the automorphisms of such algebras decompose into a graded automorphism and an inner automorphism. For the remainder of this section, $A=\kk\cQ/\cI$ is a locally gentle algebra with a unique, infinite maximal path, $\gamma$, and $\alpha_1\alpha_2\cdots \alpha_n$ is a fixed (primitive) generating cycle of $\gamma$ (i.e., $\gamma=(\alpha_1\alpha_2\cdots \alpha_n)^\infty$).

\begin{lemma}\label{lem:mbij}
    Consider the following subalgebra of $M_n(\kk[x])$,
    \begin{equation}
        M(A)\defeq \setst{P\in M_n(\kk[x])}{\substack{[P(0)]_{ii}=[P(0)]_{jj}\text{ if }s(\alpha_i)=s(\alpha_j),\\ [P(0)]_{ij}=0\text{ if }i>j}}.
    \end{equation}
    Then there exists a $\kk$-algebra automorphism $\Psi:A\rightarrow M(A)$.
\end{lemma}
\begin{proof}
    We will give an embedding of $A$ into $M_n(\kk[x])$ and show that its image is isomorphic to $M(A)$. Let $E_{ij}$ denote the matrix with $(i,j)$th entry $\delta_{ij}$ (the standard $\kk[x]$-basis elements for $M_n(\kk[x])$). 
    Considering subscripts modulo $n$, let
    \begin{equation}
        \psi(w)\defeq \begin{cases}
            \bigoplus_{\setst{i}{s(\alpha_i)=v}} E_{ii} & \text{ if } w=\epsilon_v,\\
            x^{\ell(w)}E_{i,{j+1}} &\text{ if $F(w)=\alpha_i$ and $L(w)= \alpha_j$.}
        \end{cases}
    \end{equation}
    As $\cB$ is a $\kk$-basis for $A$, this defines a $\kk$-linear map $\psi:A\rightarrow  M_n(\kk[x])$. In particular, \[\psi(\one_A)=\sum_{v\in\cQ_0}\psi(\epsilon_v)=\bigoplus_{\setst{i}{\alpha_i\in\cQ_1}}E_{ii}=I_n.\]
    To see that $\psi$ is a homomorphism of $\kk$-algebras, it suffices to check that $\psi$ respects multiplication on $\cB$. Suppose $w\in \cB_+$ with $L(w)=\alpha_i$. Let $w'\in \cB_+$. If $F(w')\neq \alpha_{i+1}$, then $ww'=0$ and $\psi(w)\psi(w')=0$. Otherwise, we have $ww'\in\cB$,  $\ell(ww')=\ell(w)+\ell(w')$, $F(ww')=F(w)$, and $L(ww')=L(w')$. Thus $\psi(ww')=\psi(w)\psi(w')$ in either case.
    As $s(w)=s(\alpha_i)$, it follows that $s(w)=v$ if and only if $\psi(\epsilon_v)\psi(w)=\psi(w)$. 
    As $\alpha_1\cdots \alpha_n$ is a cycle,  $L(w)=\alpha_j$ implies $t(w)=s(\alpha_{j+1})$, so $t(w)=v$ if and only if $\psi(w)\psi(\epsilon_v)=\psi(w)$. Lastly, it is clear that $\psi(\epsilon_i)\psi(\epsilon_j)=\delta_{ij}\psi(\epsilon_i)$.
    Thus $\psi$ is a homomorphism of $\kk$-algebras.
    
   As $A$ is locally gentle and $\gamma$ is infinite, there is exactly one path $p\in \cB$ such that $F(p)=\alpha_i$ and $\ell(p)=k$ for each $1\leq i\leq n$ and $k\geq 1$.  Since $\setst{x^kE_{ij}}{k\geq 0, 1\leq i,j\leq n}$ is a $\kk$-basis for $M_n(\kk[x])$, it follows that $\psi$ is injective. Additionally, the last arrow of $p$ must be $\alpha_{i+k-1}$, so $\psi(p)=x^kE_{i,k+i}$ (again, with indices taken modulo $n$). 
Thus $\psi(A)$ is spanned over $\kk[x^n]$ by
\begin{equation}
    \set{\bigoplus_{\setst{i}{s(\alpha_i)=v}} E_{ii}}_{v\in\cQ_0}\cup \set{x^{j-i}E_{ij}}_{1\leq i<j\leq n}\cup \set{x^{n+j-i}E_{ij}}_{1\leq j\leq i\leq n}.
\end{equation}
That is, $P\in \psi(A)$ if and only if $P_{ij}=x^{j-i}p_{ij}(x^n)$ for some $p_{ij}\in\kk[x]$ such that
(1) $p_{ii}(0)=p_{jj}(0)$ if $s(\alpha_i)=s(\alpha_j)$, (2) $p_{ij}(0)=0$ if $j<i$. Consider the diagonal matrix
\begin{equation}
    D\defeq\left(\begin{matrix}
        1 & 0 & \cdots & 0 & 0\\
        0 & x & \cdots & 0 & 0\\
        \vdots & \vdots & \ddots & \vdots & \vdots\\
        0 & 0  & \cdots & x^{n-2} & 0\\
        0 & 0  & \cdots & 0 & x^{n-1}\\
    \end{matrix}\right).
\end{equation}
Then $DE_{ij}D^{-1}=x^{i-j}E_{ij}$ and so $P\in D\psi(A)D^{-1}$ if and only if $P_{ij}=p_{ij}(x^n)$ for some $p_{ij}\in\kk[x]$ such that
 $p_{ii}(0)=p_{jj}(0)$ if $s(\alpha_i)=s(\alpha_j)$ and $p_{ij}(0)=0$ if $j<i$. It follows that $D\psi(A) D^{-1}$ (and hence $A$) is isomorphic to $M(A)$ (sending $x^nE_{ij}$ to $xE_{ij}$).
\end{proof}

The automorphism $\Psi: A\to M(A)$ this yields is given by $\Psi(\epsilon_v)= \bigoplus_{\setst{i}{s(\alpha_i)=v}} E_{ii}$ and
    \begin{equation}
        \Psi(\alpha_i)=\begin{cases}E_{i(i+1)} & \text{ if $1\leq i<n$},\\
        xE_{n1} & \text{ if $i=n$}.\\
        \end{cases}
    \end{equation}

We use $S_n$ to denote the symmetric group on $[n]=\setst{i}{1\leq i\leq n}$. For $\sigma\in S_n$, $M_\sigma\in M_n(\kk)$ denotes the corresponding permutation matrix. That is, $(M_\sigma)_{ij}=\delta_{\sigma(i)j}$. We let $\tau$ denote the $n$-cycle $(1\; 2\; \cdots \; n)$. It was shown in \cite{GLR} that for any matrix $M\in M_n(\kk[x])$, there exist unimodular matrices $U,V$ and a diagonal matrix $D$ such that $M=UDV$. In the following lemma, we show that, up to a permutation matrix, we may assume that $U(0), V(0)$ are upper triangular with ones on the diagonal (i.e., they correspond to an element of $\one+A_+$).
\begin{lemma}[Modified Smith form]\label{lem:decomp}
    For $n\in\NN$, consider the set
    \begin{equation}
        B_n\defeq \setst{P\in M_n^\times (\kk[x])}{[P(0)]_{ij}=\begin{cases} 1 & \text{ if }i=j,\\
        0&\text{ if }i>j.\end{cases}}.
    \end{equation}
    
   For any $M\in M_n(\kk[x])$, there exist $U,V\in B_n$, $\sigma\in S_n$, and a diagonal matrix $D\in M_n(\kk[x])$ such that $M=UDM_\sigma V$.
\end{lemma}

\begin{proof}
Note that $B_n$ is a multiplicative subgroup of $M_n^\times (\kk[x])$, as $U\in B_n$ if and only if $U(0)-I_n$ is strictly upper triangular. Thus, it suffices to show that for any $M\in M_n(\kk[x])$, there exist $U,V\in B_n$ such $UMV$ is product of a diagonal matrix and a permutation matrix, which we show via induction on $n$.
    Let $M\in M_n(\kk[x])$ and let $d_0\defeq \min\setst{\deg M_{ij}}{M_{ij}\neq 0}$, the minimum degree of the nonzero entries of $M$. Choose $(n-i_0,j_0)$ minimal with respect to lexicographic order such that the $(i_0,j_0)$th entry of $M$ has degree $d_0$. That is, if $(n-i_0,j_0)>(n-k,l)$, either $M_{kl}=0$ or 
    $\deg M_{kl}>d_0$. Then, there exist $u_k, v_l\in \kk[x]$ such that, for $k,l\in [n]$, $\deg (M_{kj_0}-u_kM_{i_0j_0})\leq d$ and $\deg (M_{i_0l}-v_lM_{i_0j_0})\leq d_0$ and, moreover, \begin{enumerate}
        \item[(1)] the inequality is strict if $k<i_0$ or $l>j_0$ and
        \item[(2)] $u_k,v_l\in x\kk[x]$ if $k>i_0$ or $l<j_0$.
    \end{enumerate}  Condition (1) implies that either $M_{kj_0}$ is zero or $\deg (M_{kj_0}-u_kM_{i_0j_0})<\deg (M_{kj_0})$ for any $k\in [n]$ and likewise for $M_{i_0l}$. 
    Consider matrices $U_0, V_0\in M_n(\kk[x])$ given by
        \begin{equation*}
            [U_0]_{kl}=\begin{cases}
                1 & \text{ if }k=l,\\
                -u_k & \text{ if }k\neq i_0=l\\
                0 & \text{ otherwise.}
            \end{cases}\qquad
            [V_0]_{kl}=\begin{cases}
                1 & \text{ if }k=l,\\
                -v_l & \text{ if }k=j_0\neq l,\\
                0 & \text{ otherwise.}
            \end{cases}
        \end{equation*}
    Note that $(U_0-I_n)^2=0$, as $U_0-I_n$ has its only nonzero entries in its $i_0$th column and has all zeros on the diagonal,  and thus $U_0\in M_n^\times(\kk[x])$. Similarly, $V_0-I_n$ has its only nonzero entries in its $j_0$th column and $V_0\in M_n^\times(\kk[x])$. 
    Thus condition (2) implies $U_0,V_0\in B_n$.  
    Let $M_1\defeq U_0MV_0$. Then,
        \begin{equation*}
            [M_1]_{kl} = \begin{cases} M_{kj_0}-u_kM_{i_0j_0} & \text{ if } l=j_0, k\neq i_0;\\
        M_{i_0l}-v_lM_{i_0j_0} & \text{ if }k=i_0, l\neq j_0;\\
        M_{i_0j_0} & \text{ if $k=i_0$ and $l=j_0$}.\\
            \end{cases}
        \end{equation*} 
    Let $d_1\defeq \min\setst{\deg [M_1]_{kl}}{[M_1]_{kl}\neq 0}$ and choose $(i_1,j_1)$ in the same way as $(i_0,j_0)$ was chosen (replacing $d_0$ with $d_1$). By choice of $u_k$ and $v_k$, we observe that  
    \begin{itemize}
        \item $(d_1, n-i_1, j_1)\leq (d_0, n-i_0, j_0)$
        \item If $(d_1, n-i_1, j_1)= (d_0, n-i_0, j_0)$, then $[M_1]_{i_0j}=[M_1]_{ij_0}=0$ for all $i\neq i_0$, $j\neq j_0$. 
    \end{itemize}
    So repeating this process will eventually  result  in a matrix in $M_n(\kk[x])$ which has an entry such that all other entries in that row and column are zero. Since $B_n$ is closed under multiplication, we see that there exists $U',V'\in B_n$ and $(\bar{\imath},\bar{\jmath})$ such that $(U'MV')_{\bar{\imath}\bar{\jmath}}$ is the only nonzero entry in its row and column. Let $\bar{M}$ be the matrix which results from deleting the $\bar\imath$th row and $\bar\jmath$th column of $U'MV'$. Then there exist $\bar{U},\bar{V}\in B_{n-1}$, a diagonal $\bar{D}\in M_{n-1}(\kk[x])$, and $\sigma\in S_{n+1}$ such that $\bar{U}\bar{M}\bar{V} = \bar{D}M_\sigma $.
    Let $\hat{U}\in M_n(\kk[x])$ be the matrix which results from inserting 1 into $\bar{U}$ as the $\bar\imath$th diagonal entry. That is, there exist blocks $B_{11}\in M_{\bar\imath-1}(\kk[x])$, $B_{12}\in M_{\bar\imath-1, n-\bar\imath}(\kk[x])$, $B_{21}\in M_{n-\bar\imath, \bar\imath-1}(\kk[x])$ and $B_{22}\in M_{n-\bar\imath, n-\bar\imath}(\kk[x])$ such that
    \begin{equation*}
        \bar{U}=\left(\begin{matrix}
            B_{11} & B_{12}\\
            B_{21} & B_{22}
        \end{matrix}\right) \qquad \text{ and }\hat{U}=\left(\begin{matrix}
            B_{11} & 0 &  B_{12}\\
            0 & 1 & 0\\
            B_{21} & 0 & B_{22}
        \end{matrix}\right).
    \end{equation*}
   Then, $\hat{U}U'MV' \hat{V}$ is the matrix obtained by inserting $(U'MV')_{\bar\imath\bar\jmath}$ into $DM_\sigma$ as the $\bar\imath\bar\jmath$-th entry (and zeroes elsewhere in the $\bar\imath$th row and $\bar\jmath$th column). That is, there exist blocks $C_{11}\in M_{\bar\imath-1,\bar\jmath-1}(\kk[x])$, $C_{12}\in M_{\bar\imath-1,n-\bar\jmath}(\kk[x])$, $C_{21}\in M_{n-\bar\imath,\bar\jmath-1}(\kk[x])$, and $C_{22}\in M_{n-\bar\imath,n-\bar\jmath}(\kk[x])$ such that 
   \begin{equation*}DM_\sigma=\left(\begin{matrix}
            C_{11} & C_{12}\\
            C_{21} & C_{22}
        \end{matrix}\right) \qquad \text{ and }\hat{U}U'MV' \hat{V}=\left(\begin{matrix}
            C_{11} & 0 &  C_{12}\\
            0 & (U'MV')_{\bar\imath\bar\jmath} & 0\\
            C_{21} & 0 & C_{22}
        \end{matrix}\right).
    \end{equation*}
   Since a matrix $P\in M_n(\kk[x])$ can be decomposed into a diagonal matrix and a permutation matrix exactly when it has at most one nonzero entry in each row and column, we see that there exist $\hat{D}$ and $\hat{\sigma}$ so that 
    $\hat{U}U'MV'\hat{V}=\hat{D}M_{\hat{\sigma}}$.
    Thus, since $\hat{U}U,V'\hat{V}\in B_n$ (as it is closed under multiplication), we have the desired result.
\end{proof}

\begin{example}
    To illustrate how the algorithm described above works, consider the matrix
    \begin{equation*}
        M=\left(\begin{matrix}
            6x^3-4x^2 & -3x+2 & 9x^2-4\\
            2x^2-1 & -1 & 3x+2\\
            2x^3 & -x+1 & 3x^2+2x
        \end{matrix}\right)
    \end{equation*}
    Then $(i_0,j_0)=(2,2)$ and
        \begin{equation*}
        U_0=\left(\begin{matrix}
            1 & -3x+2& 0\\
            0& 1 & 0\\
            0 & -x & 1
        \end{matrix}\right),\qquad V_0=\left(\begin{matrix}
            1 & 0& 0\\
            2x^2& 1 & 3x+2\\
            0 & 0 & 1
        \end{matrix}\right),
    \end{equation*}
so 
\begin{equation*}
        M_1\defeq U_0MV_0=\left(\begin{matrix}
            3x-2 & 0 & 0\\
            -1 & -1 & 0\\
            2x^2+x & 1 &  3x+2
        \end{matrix}\right).
    \end{equation*}
    Then $(i_1,j_1) = (3,2)$ and
     \begin{equation*}
        U_1=\left(\begin{matrix}
            1 & 0& 0\\
            0& 1 & 1\\
            0 & 0 & 1
        \end{matrix}\right),\qquad V_1=\left(\begin{matrix}
            1 & 0& 0\\
            -2x^2-x & 1 & -3x-2\\
            0 & 0 & 1
        \end{matrix}\right)
    \end{equation*}
        so 
\begin{align*} 
M_2\defeq U_1M_1V_1&=
\left(\begin{matrix}
            3x-2 & 0 & 0\\
            2x^2+x-1 & 0 & 3x+2\\
            0 & 1 & 0
            \end{matrix}\right).
\end{align*}
Thus $(\bar\imath,\bar\jmath)=(3,2)$, $U'=U_1U_0$, $V'=V_1V_0$, and 
\[\bar{M}= \left(\begin{matrix}
            3x-2 & 0 \\
            2x^2+x-1 & 3x+2\\ \end{matrix}\right).\]
Then,
    \begin{equation*}\left(\begin{matrix} 2x  & -2\\- x & 1\end{matrix}\right)
    \bar{M}
    \left(\begin{matrix} 1 &-2\\ - x & 2x+1\end{matrix}\right) 
            = \left(\begin{matrix} 0 & -1\\-1 & 0\end{matrix}\right),\end{equation*}
           so
    \begin{equation*}\left(\begin{matrix} 2x  & -2& 0\\- x & 1 & 0\\
    0 & 0 & 1\end{matrix}\right)
    U'MV'
    \left(\begin{matrix} 1 & 0 & -2\\ 0 & 1 & 0\\- x & 0& 2x+1\end{matrix}\right) 
    = \left(\begin{matrix} 0 & 0 & -1 \\-1 & 0 & 0\\ 0 & 1 & 0\end{matrix}\right).\end{equation*}
\end{example}

\begin{proposition}\label{prop:1mcase}
    Suppose that $A=\kk\cQ/\cI$ is a locally gentle algebra with a unique maximal path, $\gamma$, which is infinite. Then either $A\cong \kk[x]$ or $\Aut_0(A)\subseteq \Inn^*(A)$. Consequently, if $A\not\cong \kk[x]$, then any automorphism of $A$ decomposes into an inner automorphism and a graded automorphism.
\end{proposition}
\begin{proof}

    Suppose that $\phi\in \Aut_0(A)$, so that $\phi(w)\in w+A_{>\ell(w)}$ for each $w\in \cB$. As the center of $A$ is $\kk[m_\gamma]\cong \kk[x]$ (see \cite{FOZ}), $\phi$ restricts to an automorphism of $\kk[m_\gamma]$ and thus 
    $\phi(m_\gamma) = m_\gamma$. 
    
    Let $\Psi, M(A)$ be as in Lemma \ref{lem:mbij} and let $ n=|\cQ_1|$, so that $A\cong M(A)\subseteq M_n(\kk[x])$. As $\Psi(m_\gamma) =x I_n$, the corresponding automorphism $\Psi\circ \phi\circ \Psi^{-1}$ of $M(A)$ fixes $\kk[x]I_n$ and hence is a $\kk[x]$-algebra homomorphism of $M(A)$. Observe that $\kk[x]I_n$ is a central subalgebra of $M(A)$. Since $xM_n(\kk[x])\subseteq M(A)$, for any $P\in M_n(\kk(x))$ there exists $f(x)\in\kk[x]$ such that $f(x)P\in M(A)$ and hence the localization of $M(A)$ with respect to $\kk[x]I_n$ is $M_n(\kk(x))$. Thus the $\kk[x]$-automorphism $\Psi\circ \phi\circ \Psi^{-1}$ of $M(A)$ extends to a $\kk(x)$-automorphism of $M_n(\kk(x))$. That is,
    \begin{equation}
        \Aut_0(A)\xrightarrow[\phi\mapsto \Psi\circ \phi\circ \Psi^{-1}]{\cong}\Aut_{\kk[x]}(M(A))\hookrightarrow \Aut_{\kk(x)}M_n(\kk(x)).
    \end{equation}
    As all automorphisms of $M_n(\kk(x))$ are inner, there exists $P\in M_n(\kk(x))$ such that 
    \begin{equation}
        \Psi\circ \phi\circ \Psi^{-1}=\Delta_P.
    \end{equation}
    We may assume without loss of generality that $P\in M_n(\kk[x])$ (otherwise, there exists some polynomial $p(x)\neq 0$ such that $p(x)P\in M_n(\kk[x])$ and $\Delta_P= \Delta_{p(x)P}$, so we may use $p(x)P$ instead). 
    
    By Lemma \ref{lem:decomp}, $P=UDM_\sigma V$ for some $U,V\in B_n^\times$, $D$ diagonal, and $\sigma\in S_n$. 
    By definition of $B_n$, we have $B_n\subseteq M(A)^\times$. Moreover, as $U(0)-I_n, V(0)-I_n$ are strictly upper triangular (by definition of $B_n$), we see that $U,V\in \Psi(\one+A_+)$. That is, there exist elements $y,y'\in A_+$ such that  $\one+y, \one+y'$ are units in $A$ and $U=\Psi(\one+y),V=\Psi(\one+y')$. Note that $\Delta_{\one+y},\Delta_{\one+y'}\in \Inn^*(A)$.
    
    Consider $\hat{\phi}\defeq \Delta_{\one+y'}^{-1}\circ \phi\circ \Delta_{\one+y}^{-1}$. By choice of $y,y'$, we have 
    \begin{equation}\Psi\circ \hat{\phi}\circ  \Psi^{-1} = \Delta_{V^{-1}}\circ \Delta_P\circ \Delta_{U^{-1}}=\Delta_{DM_\sigma}.\end{equation}
    
    As $\Delta_{\one+y},\Delta_{\one+y'}\in \Inn^*(A)$, we have $\hat{\phi}\in \Aut_0(A)$. Then, since $\Psi(\alpha_i)=E_{i(i+1)}$ for $1\leq i<n$ and 
    \begin{align*}
        \left(\Psi\circ \hat{\phi}\right)(\alpha_i)= \Delta_{DM_\sigma}(E_{i(i+1)})=M_{\sigma^{-1}} D^{-1} E_{i(i+1)} DM_\sigma = D^{-1}_{ii}D_{(i+1)(i+1)} E_{\sigma(i)\sigma(i+1)},
    \end{align*}
   $\sigma$ must be the identity and $D^{-1}_{ii}D_{(i+1)(i+1)}\in 1+x\kk[x]$ for each $\leq i<n$. 
   Additionally, as $xE_{(i+1)i}\in M(A)$ and \begin{align*}
        \Delta_{DM_\sigma}(xE_{(i+1)i})=xD_{ii}D^{-1}_{(i+1)(i+1)} E_{(i+1)i}\in M(A),
    \end{align*}
    it must be that $(D^{-1}_{ii}D_{(i+1)(i+1)})^{-1}=D^{-1}_{ii}D_{(i+1)(i+1)}\in \kk[x]$ for each $\leq i<n$ as well. Thus $D^{-1}_{ii}D_{(i+1)(i+1)}=1$ for each $i<n$, so $\Delta_{DM_\sigma}=\id_{M(A)}$. Hence $\hat{\phi}=\Delta_{\one+y'}^{-1}\circ \phi\circ \Delta_{\one+y}^{-1}$ is the identity on $A$, so $\phi=\Delta_{\one+y'}\circ \Delta_{\one+y}=\Delta_{(\one+y')(\one+y)}\in \Inn^*(A).$
    
\end{proof}
\section{Automorphisms of a (locally) string algebra}\label{sec:main}

In this section, we will show that to each automorphism $\phi\in\Aut_0(A)$, there are corresponding automorphisms of $\bar{A}$ and $A^{(i)}$ (where these are as in Lemma \ref{lem:Abar}), and that we may describe $\phi$ in terms of such automorphisms. In particular, as $A^{(i)}$ is a locally gentle algebra with a unique maximal path, $\Aut_0(A^{(i)})=\Inn^*(A^{(i)})$ for each $i$.

\begin{proposition}\label{prop:decomp} Suppose that $A\not\cong \kk[x]$ is a (locally) string algebra.
    Let $\bar{A}$ be as in Lemma \ref{lem:Abar} and $\phi\in \Aut_0(A)$. Then there exists a type II derivation $\delta$ and $\psi\in \Inn^*(A)$ such that $\hat{\phi}=\exp(\delta)\circ \phi\circ \psi$ fixes $\cB\setminus \bar{A}$ and $\hat{\phi}\vert_{\bar{A}}\in \Aut_0(\bar{A})$.
\end{proposition}
\begin{proof}
    Suppose that $\phi\in\Aut_0(A)$. By Corollary \ref{cor:fixAbar}, there exists  a  type II  derivation $\delta$ such that $(\exp(\delta)\circ \phi)(\alpha)\in \langle \alpha\rangle$ for each arrow $\alpha\in\cQ_1$. 
    As $\exp(\delta)\circ \phi\in \Aut_0(A)$, we know that $(\exp(\delta)\circ \phi)(\epsilon_v)\in \langle \epsilon_v\rangle$ as well by Lemma \ref{lem:autpart}. 
    Since $\gamma_i$ are infinite maximal paths, we have 
    \begin{align*}\anni(A^{(i)}_+)&=\spn_\kk\setst{w\in \cB}{w\not\leq \gamma_i}\\&=\langle w\in\cB_{\leq 1}\mid w\not\leq \gamma_i\rangle\end{align*}
    for $1\leq i\leq m$.
    Hence $(\exp(\delta)\circ \phi)(\anni(A^{(i)}_+))\subseteq \anni(A^{(i)}_+)$ for each $i$. Since $A^{(i)}\cong A/\anni(A^{(i)}_+)$, it follows that $\exp(\delta)\circ \phi$ induces an automorphism of $A^{(i)}$. Since $A_+=\bar{A}_+\oplus\bigoplus_i A^{(i)}_+$, there there exist maps $\phi^{(i)}:\cB\rightarrow A^{(i)}_+$ and $\bar{\phi}:\cB\rightarrow \bar{A}_+$ so that
    \[(\exp(\delta)\circ \phi)(w)=w+\bar{\phi}(w)+\sum_{i=1}^m \phi^{(i)}(w).\]
    Since $A^{(i)}_+$ ($\bar{A}_+$) is an ideal of $A$, $\phi^{(i)}(w)\neq 0$ ($\bar{\phi}(w)\neq 0$) implies $w\in A^{(i)}$ ($w\in \bar{A}$) and the
    automorphism $\exp(\delta)\circ \phi$ induces on $A^{(i)}$ is $\hat{\phi}^{(i)}\in \Aut_0(A^{(i)})$, where
    \[ \hat{\phi}^{(i)}(w)=w+\phi^{(i)}(w)\]
for $w\in \cB\cap A^{(i)}$.
    Recall from Lemma \ref{lem:Abar} that $A^{(i)}$ is a locally gentle algebra with a unique maximal path, $\gamma_i$, which is infinite. If $A^{(i)}\cong \kk[x]$, then $\phi^{(i)}=0$ and we take $y_i=0$. Otherwise, $\hat{\phi}^{(i)}\in \Inn^*(A^{(i)})$ by Proposition \ref{prop:1mcase} and thus there exists $y_i\in A^{(i)}_+$ such that $\one_{A^{(i)}}+y_i\in (A^{(i)})^\times$ and $\hat{\phi}^{(i)}=\Delta_{\one_{A^{(i)}}+y_i}$. As $\one_A-\one_{A^{(i)}}\in \anni(A^{(i)})$, it follows that $\one+y_i\in A^\times$. Noting that $(\one+y_i)^{-1}-\one_A\in A^{(i)}_+$ (so $\Delta_{\one+y_i}(w)=w$ if $w\in \cB\setminus A^{(i)}$), we see that 
    $\Delta_{\one+y_i}(w)=
    w+\phi^{(i)}(w)$ for any $w\in \cB$. Let 
    \[y\defeq \prod_{i=1}^m(\one+y_i)-\one=\sum_{i=1}^m y_i.\]
    Then $\Delta_{\one+y}\in \Inn^*(A)$
    and
    $\Delta_{\one+y}(w)=w+\sum\limits_{i+1}^m\phi^{(i)}(w)$
    for $w\in\cB$. If $\hat{\phi}\defeq \Delta_{\one+y}^{-1}\circ \exp(\delta)\circ \phi$, then
    \[\hat{\phi}(w)=w+\bar{\phi}(w)\]
     for any $w\in \cB$.
    As $\bar{\phi}(w)\in \bar{A}_+$ and $\bar{\phi}(w)\neq 0$ implies $w\in \bar{A}$, we see that $\hat{\phi}$ fixes $\cB\setminus \bar{A}$ and $\hat{\phi}\vert_{\bar{A_0}}\in \Aut_0(\bar{A})$.  As $\Inn^*(A)$ is a normal subgroup of $\Aut_0(A)$, 
    \[\hat{\phi}=\Delta_{\one+y}^{-1}\circ \exp(\delta)\circ \phi=\exp(\delta)\circ \phi\circ \psi\]
    for some $\psi\in \Inn^*(A),$ so $\hat{\phi}$ is the desired automorphism.
\end{proof}

We may now apply Theorem \ref{thm:strA0} to $\bar{A}$ to conclude the following.
\begin{theorem}\label{thm:decomp} Suppose that $A\not\cong \kk[x]$ is a (locally) string algebra. Then, $\phi\in \Aut_0(A)$ may be decomposed into a composition of type I and II exponential automorphisms and $\psi\in \Inn^*(A)$. 
\end{theorem}
In the (locally) gentle  case, there are no derivations of type I. So, as type II  derivations commute, the composition of exponential automorphisms is itself a type II exponential automorphism.
\begin{proof}
Let $\phi\in \Aut_0(A)$. By Proposition \ref{prop:decomp}, there exists a  type II derivation $\delta$ and $\psi\in \Inn^*(A)$ such that $\hat{\phi}=\exp(\delta)\circ \phi\circ \psi$ fixes $\cB\setminus \bar{A}$ and $\hat{\phi}\vert_{\bar{A}}\in \Aut_0(\bar{A})$. Since $\bar{A}$ is a direct sum of string algebras and, by choice of $\delta$, we have $\hat{\phi}(\alpha)\in \langle \alpha\rangle$ for any $\alpha\in\cQ_1$ and $\hat{\phi}(\epsilon_{v})\in \langle \epsilon_v\rangle$ for any $v\in\cQ_0$, it follows from Theorem \ref{thm:strA0} there exist nilpotent derivations $\delta_i\in \Der(\bar{A})$ and $y\in \bar{A}_+$ such that 
\[\hat{\phi}\vert_{\bar{A}}=\exp(\delta_1)\circ \cdots \circ \exp(\delta_k)\circ \Delta_{\one_{\bar{A}}+y}.\] 

Since $\bar{A}_+$ kills $\cB\setminus \bar{A}$ and $\cB_0\setminus \bar{A}$ kills $\bar{A}$, we see that $\Delta_{\one_A+y}\in \Inn^*(A)$ fixes $\cB\setminus \bar{A}$ and $\Delta_{\one_A+y}\vert_{\bar{A}}=\Delta_{\one_{\bar{A}}+y}$.
We may extend $\delta_i\in \Der(\bar{A})$ to a derivation $\hat{\delta}_i$ on $A$ which kills $w\in \cB\setminus \bar{A}$. Then $\hat{\delta}_i$ is still nilpotent and that $\exp(\hat{\delta}_i)\vert_{\bar{A}}=\exp(\delta_i)$ while 
$\exp(\hat{\delta}_i)$ fixes $\cB\setminus \bar{A}$.
Thus, 
\[\hat{\phi}=\exp(\hat{\delta}_1)\circ \cdots \circ \exp(\hat{\delta}_k)\circ \Delta_{\one_{{A}}+y},\]
so 
\[\phi=\exp(-\delta)\circ \exp(\hat{\delta}_1)\circ \cdots \circ \exp(\hat{\delta}_k)\circ \psi'\]
for some $\psi'\in \Inn^*(A)$.  Moreover, since we may assume that $\delta_i$ are type I and II derivations (viewed as derivations on the string subalgebras of $\bar{A}$ corresponding to connected components of $\cQ^{(0)}$),  $\hat{\delta}_i$ will also be derivations of type I and II on $A$.
\end{proof}

As a consequence,  $\Aut(A)=\Hhat{A}\cdot \Inn^*(A)$  holds when $A$ is locally string as well. 
\begin{corollary}
    Let $A$ be a (locally) string algebra. Then  $\Aut(A)=\Hhat{A}\cdot \Inn^*(A)$.
\end{corollary}
\begin{proof}
Combining Lemma \ref{lem:aut0} and Theorem \ref{thm:decomp}, any  $f\in \Aut(A)$ may be written as a composition of a graded automorphism ($\gr f$), a composition of exponential automorphisms of type I and II, and an inner automorphism in $\Inn^*(A)$. As exponential automorphisms of type I and II belong to $H_A$, it suffices to show that a graded automorphism of $A$ must belong to $\Hhat{A}$. It was shown in \cite{JG} that for any automorphism $f$ of $A$, there exists a permutation $\sigma\in S_{\cQ_0}$ such that $f(\epsilon_v)\in \epsilon_{\sigma(v)}+A_+$. Thus any graded automorphism of $A$ acts as a permutation on its vertices, so $\gr(f)\in \Hhat{A}$.
\end{proof}

\subsection{(Locally) gentle automorphisms}
Now, we consider the special case where $A$ is (locally) gentle.
Theorem \ref{thm:decomp} tells us that $\phi\in \Aut_0(A)$ is a composition of a type II exponential automorphism and an inner automorphism. We will show that we can factor out another inner automorphism from the type II exponential automorphism to obtain a unique decomposition.

\begin{lemma}\label{lem:barg}
    Let $A$ be (locally) string. For any arrow $\alpha$, there is at most one finite maximal path which is parallel to $\alpha$ but neither starts nor ends with $\alpha$. We denote  such a path by $\bar{\gamma}_\alpha$.
\end{lemma}
\begin{proof}
    The uniqueness of $\bar{\gamma}_\alpha$, when it exists, is guaranteed by the fact that there is at most one arrow $\beta\neq \alpha$ with source $s(\alpha)$, so any path parallel to $\alpha$ which does not start with $\alpha$ must begin with $\beta$.
    Since $A$ is (locally) string, there is at most one finite maximal path with first arrow $\beta$. 
\end{proof}
The same argument applies with $L(\bar{\gamma})$. Hence, if there is a second finite maximal path $w$ parallel to $\alpha$, then $F(w)=L(w)=\alpha$. In the (locally) gentle case, $w=\alpha$.

\begin{notation}
    We will say that a derivation $\delta$ is \emph{type III} if, for any $\lambda_\alpha\in \kk$, \[\delta=\sum_{\substack{\alpha\in\cQ_1\\ \bar{\gamma}_\alpha\text{ exists}}}\lambda_\alpha\delta_{(\alpha, \bar{\gamma}_\alpha)}.\]
    If $A$ is (locally) gentle and not isomorphic to $\kk K$, then type III derivations commute and are all nilpotent of degree two.
    In that case, $\exp(\delta)$ is called a type III exponential automorphism. 
\end{notation}

\begin{lemma}\label{lem:typeIII}
    Suppose that $A$ is (locally) gentle and not isomorphic to $\kk[x]$ or $\kk K$. The type III exponential automorphisms of $A$ form a normal subgroup $E(A)$ of $H_A$ which is isomorphic to $\kk^n$ (where $n$ is the number of arrows $\alpha$ for which $\bar{\gamma}_\alpha$ exists). 
     Then  $E(A)$ is a normal subgroup of $H_A$ and has trivial intersection with $\Inn^*(A)$. Moreover, $\Aut_0(A)=E(A)\cdot  \Inn^*(A)$.
\end{lemma}

\begin{proof}
    Since type III derivations commute, it follows that type III automorphisms are closed under composition and the subgroup $E(A)$ is generated by
    \[\{\exp(\delta_{(\alpha, \lambda\bar{\gamma}_\alpha)})\}_{\substack{\alpha\in\cQ_1,\\ \bar{\gamma}_\alpha\text{ exists}}}.\]
    It follows that $E(A)\cong \kk^n$. To see that $E(A)$ is  a normal subgroup of $H_A$, let $\lambda_\alpha\in \kk$ and let 
\[\delta=\sum_{\substack{\alpha\in\cQ_1\\ \bar{\gamma}_\alpha\text{ exists}}}\lambda_\alpha\delta_{(\alpha, \bar{\gamma}_\alpha)}.\]
    Then $\exp(\delta)$ takes an arrow $\alpha$ to $\alpha+ \lambda_\alpha \bar{\gamma}_\alpha$ when $\bar{\gamma}_\alpha$ exists, and fixes all other paths. As any $\phi\in H_A$ fixes vertices and must take maximal paths to maximal paths, we see that $\phi(\bar{\gamma}_\alpha)$ must be a linear combination of maximal paths parallel to $\alpha$. As $A$ is (locally) gentle, the only such paths are $\bar{\gamma}_\alpha$ and $\alpha$. As $A\not\cong \kk K$, parallel arrows cannot both be maximal. Thus, either $\ell(\bar{\gamma}_\alpha)>1$ or $\alpha$ is not maximal. As $\phi(A_{\geq n})\subseteq A_{\geq n}$ (Lemma \ref{lem:Jautpres}), we see that $\phi(\bar{\gamma}_\alpha)$ cannot contain $\alpha$ in either case. Hence $\phi(\bar{\gamma}_\alpha)=\bar{\lambda}_\alpha \bar{\gamma}_\alpha$ for some $\bar{\lambda}_\alpha\in \kk^\times$, so
    \[\phi\circ \exp(\delta)=\exp\left(\sum_{\substack{\alpha\in\cQ_1\\ \bar{\gamma}_\alpha\text{ exists}}}\bar{\lambda}_\alpha\lambda_\alpha\delta_{(\alpha, \bar{\gamma}_\alpha)}\right)\circ \phi.\]

    Suppose that $\delta\in E(A)\cap \Inn^*(A)$. If $\alpha\not\leq \bar{\gamma}_\alpha$, then $\exp(\delta)(\alpha)=\alpha+\lambda_\alpha \bar{\gamma}_\alpha\in\langle \alpha\rangle$ implies $\lambda_\alpha=0$.  If $\alpha\leq \bar{\gamma}_\alpha$, then $\bar{\gamma}_\alpha$ must be the unique maximal path containing $\alpha$ as $A$ is (locally) gentle (i.e., $\bar{\gamma}_\alpha=\gamma_\alpha$). As $\gamma_\alpha$ is finite,  $\alpha$ occurs only once in $\gamma_\alpha,$ so  ${\gamma}_\alpha=p'\alpha p''$ for a unique pair $p', p''\in \cB_+$ such that $\alpha\not\leq p',p''$. Thus, for $x,y\in A$, the coefficient of ${\gamma}_\alpha$ in $x\alpha y$ must be the product of the coefficient of $p'$ in $x$ and $p''$ in $y$. Consequently, if ${\gamma}_\alpha$ appears with zero coefficient in $x\alpha y$, then both $p'\alpha$ and $\alpha p''$ must as well. Since both $p'$ and $p''$ are nonstationary (so $\alpha \lneq p'\alpha, \alpha p''\lneq p'\alpha p''$), there does not exist any inner automorphism $\psi$ such that $\psi(\alpha) =\alpha+\lambda \gamma_\alpha$ with $\lambda\neq 0$. Thus $\lambda_\alpha$ is zero when $\alpha\leq \bar{\gamma_\alpha}$ as well, so $\delta$ must be trivial.

To see that $\Aut_0(A)=E(A)\cdot \Inn^*(A)$, we will show that  an automorphism of type II decomposes into a type III automorphism and an inner automorphism (recall that $A$ locally gentle has no type I automorphisms). Like type III automorphisms, type II automorphisms form a commutative subgroup of $\Aut_0(A)$. This subgroup is generated by the automorphisms $\exp(\delta_{(\alpha,\lambda p)})$ where $\alpha\in\cQ_1$, $\lambda\in\kk$, and $p\in\cB\cap \epsilon_{s(\alpha)}A_{>1}\epsilon_{t(\alpha)}$ is a left-maximal and either $L(p)=\alpha$ or $p$ is right-maximal. So, it suffices to show that $\exp(\delta_{(\alpha,\lambda p)})\in \Inn^*(A)$ if $p\neq \bar{\gamma}_\alpha$.

Suppose  $L(p)=\alpha$, so $p=p'\alpha$ for some $p'\in\cB_+$. Since $A$ is (locally) gentle, $p$ left-maximal implies that $p'$ is also left-maximal. For any $w\in\cB_+$, we have $\delta_{(\alpha,\lambda p)}(w)=\lambda p'w$ and so
\[\exp(\delta_{(\alpha,\lambda p)})(w)= w+\lambda p' w = (\one+\lambda p')w(\one-\lambda p')= \Delta_{\one-p'}(w),\]
since $p'$ is left-maximal. As $p'\alpha$ is parallel to $\alpha$, we have $p'\in \epsilon_{s(\alpha)}A\epsilon_{s(\alpha)}$ and so
\[\exp(\delta_{(\alpha,\lambda p)})(\epsilon_v)= \epsilon_v= \Delta_{\one-p'}(\epsilon_v)\]
as well. Thus, since $p'\in \cB_+$, we see that $\exp(\delta_{(\alpha,\lambda p)})= \Delta_{\one-p'}\in \Inn^*(A)$.

If $p$ does not end in $\alpha$ and $p\neq \bar{\gamma}_\alpha$, then $p$ is maximal and $p=\alpha p''$ for some $p''\in\cB_+$. Similar to the above case, it follows that $\exp(\delta_{(\alpha,\lambda p)})= \Delta_{\one+p''}\in \Inn^*(A)$.
\end{proof}

\begin{lemma}\label{lem:Hgr}\cite{CEFO25}
Let $A$ be (locally) gentle and not isomorphic to $\kk[x]$ or $\kk K$. Let $C(A)$ denote the subgroup of graded automorphisms of $A$ which fix vertices and take each arrow $\alpha$ to a scalar multiple of itself.
Then  $H_A\cap \grAut(A)$ is isomorphic to $\ZZ/2\ZZ\ltimes C(A)$ if the quiver underlying $A$ is 
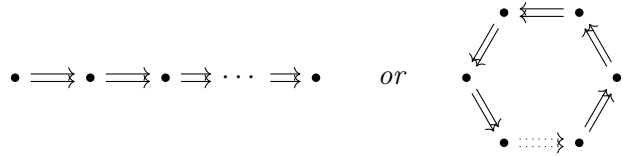
\vspace{.1cm}

    \begin{center}
    \begin{tikzpicture}
        \unvtx{1}{-7,0}
        \unvtx{2}{-6,0}
        \unvtx{3}{-5,0}
        \unvtx{4}{-3,0}
        \node (5) at (-4,0) {$\cdots$};
        \undbl{1}{2}
        \undbl{2}{3}
        \undbl{3}{5}
        \undbl{5}{4}
\node at (-2, 0) {or};
    \unvtx{0}{1,0}
    \unvtx{1}{.5,.866}
    \unvtx{2}{-.5,.866}
    \unvtx{3}{-1,0}
    \unvtx{{$k-2$}}{-.5,-.866}
    \unvtx{{$k-1$}}{.5,-.866}
    
    \draw[->,transform canvas={yshift=.03cm, xshift=.03cm}] (0) to (1);
    \draw[->,transform canvas={yshift=-.03cm, xshift=-.03cm}] (0) to (1);
    \draw[->,transform canvas={yshift=.05cm}](1) to (2);
    \draw[->,transform canvas={yshift=-.05cm}]  (1) to (2);
    \draw[->,transform canvas={yshift=.03cm, xshift=-.03cm}] (2) to (3);
    \draw[->,transform canvas={yshift=-.03cm, xshift=.03cm}] (2) to (3);
    \draw[->,transform canvas={yshift=.03cm, xshift=.03cm}](3) to ({$k-2$});
    \draw[->,transform canvas={yshift=-.03cm, xshift=-.03cm}]  (3) to ({$k-2$});
    \draw[->,dotted, transform canvas={yshift=.05cm}] ({$k-2$}) to ({$k-1$});
    \draw[->,dotted, transform canvas={yshift=-.05cm}] ({$k-2$}) to ({$k-1$});
    \draw[->,transform canvas={yshift=.03cm, xshift=-.03cm}]  ({$k-1$}) to (0);
    \draw[->,transform canvas={yshift=-.03cm, xshift=.03cm}]  ({$k-1$}) to (0);
\end{tikzpicture}
\end{center}  
and a subset of $E(A)\cdot C(A)$ otherwise.
\end{lemma}
\begin{proof}

    Suppose that $\phi\in H_A\cap \grAut(A)$. Then, $\phi$ fixes the vertices and takes $\alpha\in\cQ_1$ to $\epsilon_{s(\alpha)}A_1\epsilon_{t(\alpha)}$.
    
    Suppose $\alpha\in\cQ_1$ and $\phi(\alpha)\neq \lambda \beta$ for some arrow $\beta$. Then there must be an arrow $\alpha'\neq \alpha$ parallel to $\alpha$ such that $\phi(\alpha)= \lambda_1 \alpha+ \lambda_2\alpha'$ with both $\lambda_i\in\kk^\times$. Assume for contradiction that $\alpha'$ is not maximal. Then, there exists an arrow $\beta$ such that one of $\alpha'\beta$ and $\beta\alpha'$ is not a relation. The cases are similar, so we assume the former ($\alpha'\beta\notin\cI$). Then $\phi(\beta)= \lambda'_1\beta+\lambda'_2\beta'$ (where $\lambda'_2=0$ if there is no arrow $\beta'$ parallel to $\beta$). Since $A$ is (locally) gentle, $\alpha\beta, \alpha'\beta'\in\cI$, so 
\begin{align*}
    \phi(\alpha\beta) = (\lambda_1\lambda'_2)\alpha\beta'+(\lambda_1'\lambda_2)\alpha'\beta
\end{align*}
As $\alpha\beta'\neq \alpha'\beta\notin\cI$, we have $\lambda_1'\lambda_2=0$ and hence $\lambda_1'=0$. Then $\lambda'_2\neq 0$ so $\beta'$ exists and $\alpha\beta'\notin\cI$. However, this implies $\lambda_1\lambda'_2=0$, a contradiction. Thus, when such an $\alpha'$ exists, it must be maximal. Since $\alpha'\neq \alpha$, we see that 
$\delta_{ \alpha, -\lambda_1^{-1}\lambda_2\alpha'}$ is a type III derivation. Hence, up to an element of $E(A)$ and an element of $C(A)$, $\phi$ acts as a permutation of arrows.

Suppose $\phi\in\grAut(A)$ fixes the vertices and acts as a permutation of the arrows. Since the vertices are fixed, $\phi$ may only swap parallel arrows. Suppose $\alpha\in \cQ_1$ is not parallel to another arrow. Then $\phi$ fixes $\alpha$, and there is a maximal connected subquiver $\cQ'$ of $\cQ$ containing $\alpha$ which is fixed by $\phi$. Assume for contradiction that $\cQ\neq \cQ'$. As $\cQ$ is connected, there must exist $\beta\in \cQ_1\setminus \cQ'_1$ such that either the source or target of $\beta$ belongs to $\cQ'$.  As $\phi$ permutes the arrows and does not fix $\beta$, it must be that there exists $\beta'$ parallel to $\beta$ such that $\phi$ swaps $\beta$ and $\beta'$. Since $\cQ'$ is connected and contains an arrow and $s(\beta)$, some arrow $\bar{\beta}\in \cQ'$ must share a source or target with $\beta$. As $\beta'$ is only other arrow with source $s(\beta)$ and $\beta'\notin \cQ'$, it must be that $t(\bar{\beta})=s(\beta)$. Thus, exactly one of $\bar{\beta}\beta$ and $\bar{\beta}\beta'$ belongs to $\cI$. However, $\phi$ swaps $\bar{\beta}\beta$ and $\bar{\beta}\beta'$, yielding a contradiction. Thus, either every arrow is parallel to another, or $\phi$ fixes all arrows and $H_A\cap \grAut(A)\subseteq E(A)\cdot C(A)$.

    Now, we consider the cases where every arrow is parallel to another.
    Suppose $\cQ$ is the quiver
    \begin{center}
    \begin{tikzpicture}
        \unvtx{1}{-7,0}
        \unvtx{2}{-6,0}
        \unvtx{3}{-5,0}
        \unvtx{4}{-3,0}
        \node (5) at (-4,0) {$\cdots$};
        \draw[transform canvas={yshift=.05cm},->](1) to node[midway, above]{\footnotesize{$\alpha_1$}}(2);
    \draw[transform canvas={yshift=-.05cm},->](1) to node[midway, below]{\footnotesize{$\beta_1$}}(2);
    \draw[transform canvas={yshift=.05cm},->](2) to node[midway, above]{\footnotesize{$\alpha_2$}}(3);
    \draw[transform canvas={yshift=-.05cm},->](2) to node[midway, below]{\footnotesize{$\beta_2$}}(3);
        \undbl{3}{5}
        \draw[transform canvas={yshift=.05cm},->](5) to node[midway, above]{\footnotesize{$\alpha_k$}}(4);
    \draw[transform canvas={yshift=-.05cm},->](5) to node[midway, below]{\footnotesize{$\beta_k$}}(4);
        \end{tikzpicture}
        \end{center}
        where $k\geq 2$.
    Then, up to relabeling, $A=\kk\cQ/\langle \alpha_i\beta_{i+1}, \beta_i\alpha_{i+1}\rangle$. Let $\tau$ denote the morphism given by fixing vertices and swapping $\alpha_i$ and $\beta_i$ for each $i$. Then $\tau\in H_A\cap \grAut(A)$. Additionally, we see that this is the only permutation of arrows (other than the identity) which preserves all relations. Note that $E(A)$ is trivial in this case, and that $\tau C(A)=C(A)\tau$. 
    Thus, $H_A\cap \grAut(A)=\langle \tau\rangle C(A)\cong \ZZ/2\ZZ\ltimes C(A).$ The other case is similar, as we may assume that $\cQ$ is similarly labeled so that $\cI=\langle \alpha_k\alpha_1, \beta_k\beta_1,\alpha_i\beta_{i+1}, \beta_i\alpha_{i+1} \rangle_{1\leq i<k}$ and $\tau$ is again the only nontrivial permutation of the arrows which preserves the relations. Again, $\tau C(A)=C(A)\tau$, so we have the same result.
\end{proof}

In all cases $C(A)\cong (\kk^\times)^{|\cQ_1|}$ in all cases. We can now compute $H_A/H_A\cap \Inn^*(A)$.

\begin{theorem}\label{thm:MIgentle}
    Let $A\not\cong \kk[x]$ be (locally) gentle. Then  
    $H_A/H_A\cap \Inn^*(A)$ is isomorphic to $\GL_2(\kk)$ if $A\cong \kk K$, to $\ZZ/2\ZZ\ltimes (\kk^\times)^{|\cQ_1|}$ if the underlying quiver is
 
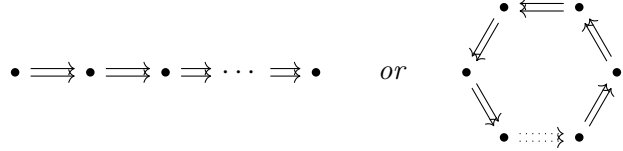
\begin{center}
    \begin{tikzpicture}
        \unvtx{1}{-7,0}
        \unvtx{2}{-6,0}
        \unvtx{3}{-5,0}
        \unvtx{4}{-3,0}
        \node (5) at (-4,0) {$\cdots$};
        \undbl{1}{2}
        \undbl{2}{3}
        \undbl{3}{5}
        \undbl{5}{4}
\node at (-2, 0) {or};
    \unvtx{0}{1,0}
    \unvtx{1}{.5,.866}
    \unvtx{2}{-.5,.866}
    \unvtx{3}{-1,0}
    \unvtx{{$k-2$}}{-.5,-.866}
    \unvtx{{$k-1$}}{.5,-.866}
    
    \draw[->,transform canvas={yshift=.03cm, xshift=.03cm}] (0) to (1);
    \draw[->,transform canvas={yshift=-.03cm, xshift=-.03cm}] (0) to (1);
    \draw[->,transform canvas={yshift=.05cm}](1) to (2);
    \draw[->,transform canvas={yshift=-.05cm}]  (1) to (2);
    \draw[->,transform canvas={yshift=.03cm, xshift=-.03cm}] (2) to (3);
    \draw[->,transform canvas={yshift=-.03cm, xshift=.03cm}] (2) to (3);
    \draw[->,transform canvas={yshift=.03cm, xshift=.03cm}](3) to ({$k-2$});
    \draw[->,transform canvas={yshift=-.03cm, xshift=-.03cm}]  (3) to ({$k-2$});
    \draw[->,dotted, transform canvas={yshift=.05cm}] ({$k-2$}) to ({$k-1$});
    \draw[->,dotted, transform canvas={yshift=-.05cm}] ({$k-2$}) to ({$k-1$});
    \draw[->,transform canvas={yshift=.03cm, xshift=-.03cm}]  ({$k-1$}) to (0);
    \draw[->,transform canvas={yshift=-.03cm, xshift=.03cm}]  ({$k-1$}) to (0);
\end{tikzpicture}
\end{center}      
    or, otherwise, $(\kk^\times)^{|\cQ_1|}\ltimes\kk^n$ (where $n$ is the number of arrows $\alpha$ for which $\bar{\gamma}_\alpha$ exists).
\end{theorem}
\begin{proof}
    It is known that $\Aut(\cong \kk K)\cong \GL_2(\kk)$, viewing $\kk K_1$ as a $\kk$-vectorspace. Otherwise,
    \[\Aut(A)=\grAut(A)\cdot \Aut_0(A)=\grAut(A)\cdot E(A)\cdot \Inn^*(A),\]
    and $E(A)$ is a normal subgroup of $H_A$ with trivial intersection with $\Inn^*(A)$. Since $\grAut(A)\cap \Inn^*(A)$ is also trivial (as $\Inn^*(A)\subseteq \Aut_0(A)$), we see that 
    \[H_A/H_A\cap \Inn^*(A)\cong (H_A\cap \grAut(A))\cdot E(A).\]
    In the other two cases where all edges are doubled, $E(A)$ is trivial and
    \[H_A/H_A\cap \Inn^*(A)\cong H_A\cap \grAut(A) \cong \ZZ/2\ZZ\ltimes (\kk^\times)^{|\cQ_0|}.\]
    Otherwise,
$H_A/H_A\cap \Inn^*(A)\cong  C(A)\cdot E(A)\cong (\kk^\times)^{|\cQ_1|}\ltimes \kk^n.$
\end{proof}

\subsection*{Acknowledgments} The author would like to thank Jason Gaddis,  Ken Goodearl, Birge Huisgen-Zimmmerman, Amrei Oswald, and James Zhang for helpful contributions.
S. Ford was supported in part by the US National Science Foundation (grant Nos. DMS-2302087 and  DMS-2001015).

\printbibliography
\end{document}